\documentclass[11pt]{article}
\usepackage{amsmath,amssymb,amsthm,nicefrac,mathrsfs}
\usepackage[colorlinks]{hyperref}

\usepackage[a4paper, margin=4cm, footskip=1cm]{geometry}

\newtheorem{theorem}{Theorem}[section]
\newtheorem{lemma}[theorem]{Lemma}
\newtheorem{proposition}[theorem]{Proposition}

\newtheorem{assumption}[theorem]{Assumption}
\newtheorem{remark}[theorem]{Remark}
\newtheorem{example}[theorem]{Example}
\numberwithin{equation}{section}

\newcommand{\be}{\begin{equation}}
\newcommand{\ee}{\end{equation}}

\newcommand{\bes}{\begin{equation*}}
\newcommand{\ees}{\end{equation*}}

\def\E{\bE}
\def\T{\bT}
\def\P{\bP} 

\def\cC{\mathcal{C}}

\def\bE{\mathbb{E}}

\def\bP{\mathbb{P}}

\def\bR{\mathbf{R}}
\def\bZ{\mathbf{Z}}

\newcommand{\Z}{\mathbf{Z}}

\newcommand{\R}{\mathbf{R}}

\renewcommand{\d}{{\rm d}}

\renewcommand{\geq}{\geqslant}
\renewcommand{\leq}{\leqslant}
\renewcommand{\ge}{\geqslant}
\renewcommand{\le}{\leqslant}

\newcommand{\lip}{\text{\rm Lip}_\sigma }

\renewcommand{\P}{\mathrm{P}}

\def\ind{\mathbf{1}}

\def\m1{\mathbf{1}}

   \DeclareMathOperator{\Cov}{Cov}  

 \def\bT{\mathbb{T}}

\allowdisplaybreaks[1]

\allowdisplaybreaks[1]

\author{Mohammud Foondun\\University of Strathclyde \and Mathew Joseph\\University of Sheffield\\
\and
Shiu-Tang Li\\ University of Utah}
\title{An approximation result for a class of stochastic heat equations with colored noise
\date{}
}

\begin{document}

\maketitle

\begin{abstract}
We show that a large class of stochastic heat equations can be approximated by systems of interacting stochastic differential equations. As a consequence, we prove various comparison principles extending earlier works of \cite{muel} and \cite{jose-khos-muel} among others. Among other things, our results enable us to obtain sharp estimates on the moments of the solution. A main technical ingredient of our method is a local limit theorem which is of independent interest.

\end{abstract}

\section{Introduction}
Consider the following stochastic heat equation,
\be \label{eq:she}
\frac{\partial }{\partial t} u_t(x) =-\nu(-\Delta)^{\alpha/2} u_t(x) +\sigma\big(u_t(x)\big)\dot{F}(t,x)\, \quad t\ge 0,\, x\in \bR^d,
\ee
with bounded and non-negative initial condition  $u_0(\cdot)$. The operator $-\nu(-\Delta)^{\alpha/2}$ is the generator of a strict stable process. The function $\sigma: \R \to \R$ is assumed to be a Lipschitz continuous function, that is $|\sigma(x)-\sigma(y)| \le \lip |x-y|$ for all $x,y$ and some constant $\lip>0$. The noise term  $\dot{F}(t,x)$ is a Gaussian random field satisfying 
\bes
\Cov\big(\dot{F}(t,x),\, \dot{F}(s,y)\big) = \delta_0(t-s) f(x-y).
\ees
We will assume that the spatial correlation is given by the {\it Riesz kernel}:
\be \label{eq:f:riesz}
f(x,y)= |x-y|^{-\beta}, \quad 0< \beta<d.
\ee
We follow Walsh \cite{wals} to make sense of \eqref{eq:she} via the following integral equation,
\bes
u_t(x) = (p_t*u_0)(x) + \int_0^t \int_{\R^d} p_{t-s} (x-y) \sigma\big(u_s(y)\big) F(\d s \d y),
\ees
where $p_t(x)$ is the density of the strict stable process with generator $-\nu(-\Delta)^{\alpha/2}$.

Questions of existence and uniqueness of \eqref{eq:she} under the above conditions are settled. For the unique solution $u_t$ to satisfy 
\begin{align} \label{eq:u:mom}
\sup_{x\in \R^d, \, t\in [0,\,T]}\E|u_t(x)|^m<\infty\quad\text{for all}\quad m\geq 2\quad\text{and}\quad T<\infty,
\end{align}
we will further need
\be \label{eq:be:al}
\beta < \alpha.
\ee
In particular this implies that $\beta < \min(\alpha, d)$, a condition which we will assume throughout the paper.  We refer  the reader to \cite{chen-dala}, \cite{spde-mini}, \cite{dala},  \cite{foon-khos-09}, \cite{foon-khos-col} and \cite{wals} for proofs and technical details, for both white and colored noise driven equations.  A detailed study of properties of such equations can be found in \cite{foon-khos-09} and \cite{foon-khos-col}; see also the extensive bibliography in these articles. When $\sigma(u)\propto u$, \eqref{eq:she} is referred as the Parabolic Anderson model (PAM). This equation is related to the KPZ equation via the Hopf-Cole tranform, see \cite{hair}. In this case, one can use the Feynman-Kac representation of the solution to make a very detailed analysis of the moments of the solution; see the work of X. Chen in \cite{chen-15} and \cite{chen-16}. Getting sharp estimates on moments of solutions is usually a starting point for a deep understanding of the almost sure pathwise behavior of the solutions. Therefore, whether one can get similar information about the moments for our more general equation is an important question.  As a consequence of the main result of this article, we will give a positive answer to the above for colored-driven equations, extending the result in \cite{jose-khos-muel} where this question was answered positively for white noise driven equations.

Consider the following natural approximation of \eqref{eq:she} by a system of interacting stochastic differential equations. For $x\in \epsilon \bZ^d$, let
\be \label{eq:disc:approx}
\d U_t^{(\epsilon)}(x) = \big(\mathscr{L}^{(\epsilon)} U_t^{(\epsilon)}\big)(x)\, \d t +  \frac{1}{\epsilon^d} \sigma\big(U_t^{(\epsilon)}(x)\big)\, \d B^{(\epsilon)}_t(x).
\ee

The operator $\mathscr{L}^{(\epsilon)}$ is an appropriate generator of a continuous time random walk $X^{(\epsilon)}$ on the fine lattice $\epsilon \Z^d$. These random walks are of the form $X^{(\epsilon)}_t=\epsilon X_{t/\epsilon^{\alpha}}$ where $X$ is a continuous time rate $1$ random walk on $\Z^d$. 
The Brownian motions $B^{(\epsilon)}_t(\epsilon k),\, k\in \Z^d$ are defined by the following,
\bes
B_t^{(\epsilon)}(\epsilon k) := \int_0^t\int_{C^{(\epsilon)}(\epsilon k)} F(\d s\, \d y),
\ees
where $C^{(\epsilon)}(\epsilon k):=\{x:=\epsilon k+y: -\frac{\epsilon}{2}\leq y_i< \frac{\epsilon}{2} \}$.  It is then easy to verify that 
\be \label{eq:bm:cov}
\Cov\big(B^{(\epsilon)}_t(\epsilon k), \, B^{(\epsilon)}_s(\epsilon l)\big)= \min(s,t)\cdot \int_{C^{(\epsilon)}( \epsilon k)}\int_{C^{(\epsilon)}( \epsilon l)} f(x-y) \,\d x\, \d y.
\ee
The initial condition $U_0^{\epsilon}(x)$, is defined by 
\begin{align*}
U_0^{\epsilon}(x):=\frac{1}{\epsilon^d}\int_{C^{(\epsilon)}(x)}u_0(y)\,\d y\quad\text{with}\quad x\in \epsilon \Z^d.
\end{align*}

Our main theorem \ref{thm:main} shows that for a large class of random walks, the approximations \eqref{eq:disc:approx} converge in a strong sense to \eqref{eq:she}.  There are a lot of papers dealing with approximations of SPDEs. These include \cite{shi-shi} and \cite{funa} which partially motivated the work in \cite{jose-khos-muel}.  Our main result is most similar to Theorem 2.4 of the latter paper where only white noise driven equations were considered. Here we are providing significant extensions of the results in \cite{jose-khos-muel}. The fact that we are considering noises which are spatially correlated makes the proof much harder. We will require some new ideas, one of which is a significant extension of the  {\it local limit theorem} for stable processes. To the best of our knowledge, this result is new and is of independent interest.  In \cite{jose-khos-muel}, the approximations were driven by independent Brownian motions. A major difference here is that our approximating SDEs are now driven by correlated Brownian motion. This is a source of additional technical hurdles. A section is devoted to the study of these SDEs. 

Our main theorem is stated for noises with correlation function given by the Riesz kernel but the reader will soon discover that this is not a major restriction. In fact, our choice of this particular kernel was partly motivated by two difficulties that the Riesz kernel presents; it has a singularity at the origin and it has a fat tail.  Another reason for this choice is that Riesz kernel represents an interpolation between smooth and white noise. All of our results will therefore hold whenever the correlation function is {\it nicer}; see the final section of this paper.

 We now describe the main results with some more care. But before, let us introduce some notations. Let $\mu$ be the \emph{dislocation distribution} of the continuous time, rate one random walk $X$ with generator $\mathscr{L}$. This is the distribution of $X_\gamma$ where
\bes
\gamma= \inf\{t\ge 0: X_t \ne 0\}.
\ees
The characteristic function of $\mu$ will be denoted by
\bes
\hat\mu(z) = \sum_{k\in \bZ^d} e^{iz\cdot k} \mu(k)= \E \exp(iz\cdot X_\gamma), \quad z \in \T^d= (-\pi, \pi]^d.
\ees
Our main assumptions on the random walk will be stated in terms of $\hat\mu$. Let functions $\mathcal{D}(z)$ and $\mathcal{E}_{k}(z)$ be functions defined by 
\be\begin{split} \label{eq:muhat}
\hat\mu(z)&= 1- \nu |z|^\alpha + \mathcal{D}(z) \\
\frac{\partial^k}{\partial z_i^k}\hat\mu(z)& = -\nu \frac{\partial^k}{\partial z_i^k}|z|^\alpha+ \mathcal{E}_{k,i}(z),\quad\text{for}\quad k\geq 1, \, 1\le i\le d. \\
\end{split}\ee
We shall primarily be working under the following assumptions.
\begin{assumption} \label{cond1} Assume that $\{z \in (-\pi,\pi]^d: \hat\mu(z) =1\} =\{0\}$ and that there exists $0<a<1$ such that
\be\begin{split} \label{eq:muhat:zero}
\mathcal{D}(z) & =  O(|z|^{a+\alpha})\quad\quad \text{as} \quad |z|\to 0.\\
\end{split}\ee
\end{assumption}
\begin{assumption}
\label{cond2} Away from zero, 
$\mathcal{D}(z)$ is $k+3$ times continuously differentiable and for all $k\le d+3$
\be \label{eq:cond2}
\mathcal{E}_{k,i}(z) = O(|z|^{\alpha+a-k}) \quad \text{as} \quad|z|\to 0.
\ee
\end{assumption}
Under Assumption \ref{cond1},  the walk is in the domain of attraction of a strict Stable($\alpha$) process. Assumption \ref{cond2} allows us to give a good decay rate 
in $x$ in the local limit theorem which compensates for the fat tails of the Riesz kernels in our approximation Theorem \ref{thm:main}; see Proposition \ref{prop:x} below.

The frequently used notation $[x]$ for a vector $x=(x_1,x_2,\cdots, x_d) \in \R^d$ denotes the point $([x_1],[x_2],\cdots, [x_d]) \in \Z^d$, where $[a], a\in \R$ is the largest integer smaller than or equal to $a$. For functions $f$ and $g$, we say $f(x)\lesssim g(x)$ if there exists a constant $C$ independent of $x$ such that $f(x)\le Cg(x)$. Our first result gives a rate of convergence of the moments of $U_t^{(\epsilon)}$ to those of $u_t$.

\begin{theorem} \label{thm:main} Let $u_t(x)$ and $U^{(\epsilon)}_t(x)$ be the unique solutions to \eqref{eq:she} and \eqref{eq:disc:approx} respectively.  Suppose that Assumptions \ref{cond1} and \ref{cond2} hold. Then for any positive $\rho \leq \min((\alpha-\beta)/2, \, a)$,
\be \label{eq:ass}
\E\big[\big\vert U_t^{(\epsilon)}(\epsilon [x/\epsilon])- u_t(x)\big\vert^m\big] \lesssim  \left(\frac{\epsilon}{t^{1/\alpha}}\right)^{a m}+ \epsilon^{\rho m},
\ee
uniformly for $x\in \R^d$ and  $\epsilon^\alpha\le t\le T$.
\end{theorem}
\begin{remark}The first term on the right of \eqref{eq:ass} gives the rate of convergence in the deterministic part of \eqref{eq:disc:approx} while the second term gives the rate of convergence of the stochastic part.
\end{remark}
At this point, it is natural to ask whether there are random walks satisfying both Assumptions \ref{cond1} and \ref{cond2}. At the end of Section \ref{sec:approx:stable}, we give an affirmative answer to this question by providing a fairly large family of random walks satisfying these conditions.  We point out two other improvements over Theorem 2.4 of \cite{jose-khos-muel}. We have managed to better the rate of convergence; the result for the white noise case in \cite{jose-khos-muel} should be the thought of as the case corresponding to $\beta=1$, which only makes sense in dimension $1$.  We are also able to handle more general initial conditions.  Some more work shows the following weak convergence result.

\begin{theorem} \label{thm:strong:approx}Suppose that Assumptions \ref{cond1} and \ref{cond2} hold, then for all $M>0$, $0<t_0<T$ and $\rho<\min((\alpha-\beta)/2, a)$, 
\bes
\epsilon^{-\rho} \sup_{t\in [t_0,T]} \;\sup_{x\in \R^d:\, \|x\|\le \epsilon^{-M}} \big\vert U_t^{(\epsilon)}(\epsilon[x/\epsilon])-u_t(x)\big\vert
 \stackrel{P}{\to} 0 \quad \text{as} \quad \epsilon \downarrow 0.
\ees
\end{theorem}
Our first application is an extension of the fundamental pathwise comparison principle of Mueller \cite{muel}. The comparison principle is one of the few general results in stochastic PDEs. It is very useful. In certain cases we can replace a general initial profile $u_0$ bounded away from $0$ and infinity by a constant profile if we are interested in studying the large-time or large-space asymptotics for the equation; see for instance \cite{conu-jose-khos} and \cite{conu-jose-khos-shiu}. See also the recent preprint \cite{chen-kunw} which deals with comparison principles for white noise driven stochastic heat equations, and the references therein. An argument, different from ours, to prove the following theorem was outlined in \cite{chen-16}. Additional information about comparison principles for stochastic differential equations can be found in the forthcoming thesis \cite{STL}.
\begin{theorem} \label{thm:comp} Let $u$ and $v$ be solutions to \eqref{eq:she} with initial profiles $u_0$ and $v_0$ respectively, and
such that $u_0(x) \le v_0(x)$ for all $x\in \R^d$. Then
\bes
\bP\left[u_t(x) \le v_t(x) \text{ for all } x\in \R^d, \, t \ge 0\right]=1.
\ees
\end{theorem}
As mentioned earlier, one of the main results of this paper is the following moment comparison principle.  
\begin{theorem}\label{thm:mom:comp} Let $u$ be the solution to \eqref{eq:she} and $\bar u$ be the solution to
\eqref{eq:she} but with $\sigma$ replaced by another Lipschitz continuous function $\bar \sigma$ such that $\sigma(0)=\bar\sigma(0)=0$. Assume that 
$\sigma(x)\ge \bar \sigma(x)\ge 0$ holds for all $x \in \R_+$. Then for any integer $m\ge 1$ and $k_1,k_2,\cdots, k_m\in \mathbf{N},\, x_1, x_2, \cdots, x_m \in \R^d, \, t\ge 0$ we have
\bes
\E\big[u_{t_1}(x_1)^{k_1}\cdots u_{t_m}(x_m)^{k_m}\big] \ge \E\big[\bar u_{t_1}(x_1)^{k_1}\cdots \bar u_{t_m}(x_m)^{k_m}\big].
\ees
\end{theorem}
A key phenomenon exhibited by many equations of the type \eqref{eq:she} is intermittency. This happens when there are some rare regions in space and time at which the solution is extremely large. Mathematically, this phenomenon is analyzed using moment Lyapunov exponents, see \cite{carm-molc}, \cite{foon-khos-09}, \cite{foon-khos-col} and \cite{geor-jose-khos-shiu}.
 As a result of the theorem above we could provide bounds on the Lyapunov exponents of \eqref{eq:she} if we could compare it to an equation for which bounds are already known. The moments of the Parabolic Anderson model have been very carefully analysed for a large class of equations; see for instance \cite{bert-canc} and \cite{chen-16}. Therefore, a consequence of Theorem \ref{thm:mom:comp} is the following. 

\begin{theorem} \label{thm:lyap} Suppose there exists positive constants $0<a<b$ such that $a\le u_0(x) \le b$ for all $x$, and suppose that $\sigma(0)=0$ and there is a $\ell_\sigma>0$ such that $\sigma(x)\geq \ell_{\sigma} |x| $ for all $x\in \R^d$. Then,
\bes
m^{\frac{2\alpha-\beta}{\alpha-\beta}} \lesssim \liminf_{t \to \infty} \frac{1}{t}\, \log \E\big(\big\vert u_t(x) \big\vert^m\big) \le \limsup_{t \to \infty} \frac{1}{t}\, \log \E\big(\big\vert u_t(x) \big\vert^m\big) \lesssim   m^{\frac{2\alpha-\beta}{\alpha-\beta}}.
\ees
\end{theorem}
Under the conditions of the above theorem one can also give bounds on the upper and lower exponential growth indices considered in \cite{huan-le-nual}.

We end this introduction with a plan of the paper. In Section \ref{sec:int_sys} we consider a system of interacting SDEs driven by correlated Brownian motions and prove a moment comparison principle for the system. After than  in Section \ref{sec:approx:stable} we prove a local limit theorem needed for Theorem \ref{thm:main} which we prove in Section \ref{sec:thm:main}. Section \ref{sec:remain} proves the remaining results stated in the introduction. Finally in Section \ref{sec:extend} we state several extensions to our results.

For a random variable $Z$, we denote $\|Z\|_p:= \E[|Z|^p]^{1/p}$. Throughout this paper $C$ will denote an arbitrary constant which might change from line to line. We will use the symbols $c_1,c_2,\cdots $ to denote constants whose value remains fixed throughout a proof but might be different in a different proof. 

\section{Interacting systems of stochastic differential equations} \label{sec:int_sys}
In this section, we study a class of interacting SDEs which we will use to approximate the SPDE. The results in this section are inspired by \cite{cox-flei-grev}. As opposed to \cite{cox-flei-grev}, which deals with independent driving Brownian motions, here the driving Brownian motions are correlated. It is worth pointing out that one of the main motivations behind \cite{cox-flei-grev} was to give general ergodic theoretic results for such systems by comparing them with ``exactly solvable" models such as the Fisher-Wright model and the Feller's branching diffusion model. 

We will need to first prove some basic results concerning our system. Consider 
\be \label{eq:disc:she}
\d U_t(x) = \big(\mathscr{L}U_t\big)(x) + \sigma\big(U_t(x)\big) \d B_t(x),\quad\text{where}\quad x\in \Z^d.
\ee
The operator $\mathscr{L}$ is the generator of a continuous time random walk $X_t$ defined by
\be \label{eq:gen}
\mathscr{L}:=\nu \sum_{i,j\in \Z^d} (p_{i,j} - \delta_{i,j})z_j \frac{\partial}{\partial z_i},
\ee
for some probability transition function $p_{i,j} = p(j-i)$. Here $\sigma:\R\to \R$ is a Lipschitz function and $B_t(x),\, x\in \Z^d$ denote a collection of correlated Brownian motions with 
\be \label{eq:bm:corr}
\Cov\big(B_s(x),\, B_t(y)\big)=(s\wedge t)\cdot \mathcal{R}(x-y),
\ee
where $\mathcal{R}$ is a nonnegative definite and symmetric function satisfying
\be \label{eq:ass:dex}
\int_0^t \d s \, \sum_{y,z\in \Z^d} P_{s}(y) \mathcal{R}(y-z) P_{s}(z) < \infty \quad \text{ for all } t>0,
\ee
with
\bes
P_t(x) : = \P(X_t=-x), \quad x \in \Z^d.
\ees
By a solution to \eqref{eq:disc:she} with a bounded initial profile $U_0: \Z^d\to \R$, we mean a solution to the following integral equation 
\bes
U_t(x) = \big(P_t* U_0\big) (x) + \int_0^t \sum_{y\in \Z^d} P_{t-s}(x-y) \sigma\big(U_s(y)\big)\, \d B_s(y), \quad x \in \Z^d
\ees
where
\bes
\big(P_t* U_0\big) (x) := \sum_{y\in \Z^d} P_{t}(x-y) U_0(y).
\ees
The main result of this section is the following.
\begin{theorem} \label{thm:mom:disc:she}
Consider two systems of equations of the form \eqref{eq:disc:she} but with $\sigma_1$ and $\sigma_2$ instead of $\sigma$. Let $U_t$ and $V_t$ be the unique respective solutions. Suppose $\sigma_1(x)\ge \sigma_2(x)\ge0$ for all $x\in \R_+$ with $\sigma_1(0)=\sigma_2(0)=0$. Then for any $x_1,\,x_2,\cdots, x_m\in \Z^d,\, t_1,\, t_2\cdots t_m \ge 0$ and nonnegative integers $k_1,\, k_2,\, \cdots k_m$,
\be \label{eq:mom:disc:she}
\E\big[U_{t_1}(x_1)^{k_1}\cdots U_{t_m}(x_m)^{k_m}\big] \ge \E\big[V_{t_1}(x_1)^{k_1}\cdots V_{t_m}(x_m)^{k_m}\big].
\ee
\end{theorem}

The complete proof of this result is lengthy but the main underlying idea is quite simple. Consider the following stochastic differential equations
\begin{align*}
\d X_t = b(X_t) \d t+ \sigma_1\big(X_t\big) \d B_t,
\end{align*} 
and
\begin{align*}
\d Y_t = b(Y_t) \d t+ \sigma_2\big(Y_t\big) \d B_t,
\end{align*} 
with the same initial condition $x_0$. Set
\bes
P_t^{\sigma_1} f(x):=\E^xf(X_t) \quad\text{and}\quad 
P_t^{\sigma_2} f(x):=\E^xf(Y_t),
\ees
and let $\mathcal{L}^{\sigma_1},\, \mathcal{L}^{\sigma_2}$ be the generators 
corresponding to $X_t$ and $Y_t$ respectively. The idea is to show that
\begin{align}\label{i1}
P_t^{\sigma_1} f(x)\ge P_t^{\sigma_2} f(x),
\end{align}
whenever  $\sigma_1\ge \sigma_2$ and $f$ belonging to some appropriate class of functions. By appealing to the following ``integration by parts" formula
\begin{equation*}
P_t^{\sigma_1} f(x)-P_t^{\sigma_2} f(x)=\int_0^{t}P^{\sigma_2}_{t-s}(\mathcal{L}^{\sigma_1}-\mathcal{L}^{\sigma_2})P_t^{\sigma_1} f(x)\,\d s,
\end{equation*}
showing \eqref{i1} amounts to proving
\begin{equation*}
(\mathcal{L}^{\sigma_1}-\mathcal{L}^{\sigma_2})P_t^{\sigma_1} f(x)\geq 0.
\end{equation*}
This is the strategy used in \cite{cox-flei-grev} and which we will adopt here. Since our equations are significantly more complicated, we will need to overcome a few technical difficulties. We begin with the following existence-uniqueness result. Since the ideas involved are quite standard, we will only give a sketch of the proof.

\begin{theorem} \label{thm:dis:mom:bd}
 Fix $T>0$ and let $U_0: \Z^d \to \R$ be a bounded initial function. Under the assumption \eqref{eq:ass:dex}
there exists a unique
solution to \eqref{eq:disc:she} such that
\be \label{eq:dis:mom:bd}
\sup_{0\le t\le T}\sup_{x  \in \Z^d} \E \big[\big \vert U_t(x)\big \vert^m\big] < \infty \quad \text{ for all } m \ge 2.
\ee
\end{theorem}
\begin{proof}The proof of existence and uniqueness uses the standard Picard iteration scheme. Define iteratively
\bes
U_t^{(n+1)}(x)= \big(P_t* U_0\big) (x) + \int_0^t \sum_{y} P_{t-s}(x-y) \sigma\big(U^{(n)}_s(y)\big)\, \d B_s(y).
\ees
For $\beta>0$, set
\bes
\mathcal{A}_{\beta, n+1} := \sup_{x \in \Z^d} \sup_{t\ge 0} e^{-\beta t} \cdot \big\| U^{(n+1)}_t(x) - U^{(n)}_t(x)\big\|_m^2
\ees
We now use Burkholder's inequality and the fact that noise is correlated to obtain
\bes
 \mathcal{A}_{\beta, n+1} \le  \lip^2\cdot  \mathcal{A}_{\beta, n}  \int_0^t \d s \, e^{-\beta s} \sum_{y,z} P_{s}(x-y) \mathcal{R}(y-z) P_{s}(x-z).
\ees
Assumption \eqref{eq:ass:dex} allows us to choose $\beta$ large enough so that
\bes
\mathcal{C}:=\lip^2 \int_0^t \d s \, e^{-\beta s} \sum_{y,z} P_{s}(x-y) \mathcal{R}(y-z) P_{s}(x-z) <1,
\ees
one gets $\mathcal{A}_{\beta, n+1} \le \mathcal{A}_{\beta,1} \cdot \mathcal{C}^n$ and $\mathcal{A}_{\beta, n+1}$ decreases
exponentially fast to $0$. The completeness of $L^p(\P)$ gives the existence of a solution satisfying \eqref{eq:dis:mom:bd}. Uniqueness follows from a standard argument.
\end{proof}

In order to prove Theorem \ref{thm:mom:disc:she} we need several approximation results which are interesting in their
own right. Our first result shows that one can approximate $\sigma$ by a function which has bounded support. 
For any $N\geq 1$, let $\sigma^{(N)}$ be a Lipschitz function defined as follows
\bes 
\sigma^{(N)}(x)= \begin{cases}
\sigma(x), &   x \in [-N,N], \\
0, & x \notin (-2N,2N), \\
\sigma(N)\cdot\left\lbrace \frac{2N-x}{N}\right\rbrace, &  x \in (N,2N), \\
\sigma(-N) \cdot \left\lbrace\frac{x+2N}{N}\right\rbrace,&  x \in (-2N,-N).
\end{cases}
\ees
We note that the Lipschitz constant is independent of $N$.

\begin{proposition} \label{prop:sigma:comp} Consider \eqref{eq:disc:she} but with $\sigma^{(N)}$ instead of $\sigma$ and let $U^{(N)}$ denote its unique solution. Then $\lim_{N\to \infty} U^{(N)}_t(x)=U_t(x)$ in $L^m(\P)$ for every $m\ge 2$.
\end{proposition}
\begin{proof}
We begin by writing 
$U_t^{(N)}(x) - U_t(x) = \mathcal{T}_1+\mathcal{T}_2$,
where
\be\begin{split} \label{eq:t1t2}
\mathcal{T}_1&= \int_0^t \sum_y P_{t-s}(x-y)\cdot \big[\sigma^{(N)} \big(U_s^{(N)}(y)\big)-\sigma\big(U_s^{(N)}(y)\big)\big] \, \d B_s(y)\\
\mathcal{T}_2&= \int_0^t \sum_y P_{t-s}(x-y)\cdot \big[\sigma\big(U_s^{(N)}(y)\big)-\sigma\big(U_s(y)\big)\big] \, \d B_s(y).
\end{split}\ee
We begin with the first integral. From the definition of $\sigma^{(N)}$, we have 
\begin{align*}
\E\big[|\sigma^{(N)} \big(U_s^{(N)}(y)\big)-\sigma\big(U_s^{(N)}(y)\big)|^m\big] &=
\E\big[|\sigma^{(N)} \big(U_s^{(N)}(y)\big)-\sigma\big(U_s^{(N)}(y)\big)|^m1_{\{U_s^{(N)}(y)\ge N\}}\big]\\
&\lesssim N^m\P \big(\big| U_s^{(N)}(y)\big |\geq N\big) \lesssim\frac{1}{N^m},
\end{align*}
where we have used Chebyshev's inequality and uniform (in $N$) bounds on the moments of $U^{(N)}$. The existence of these uniform bounds are justified by the fact that the Lipschitz coefficients of $\sigma^{(N)}$ are bounded above by a constant independent of $N$. An application of Burkholder's inequality gives
\begin{align*}
\|\mathcal{T}_1\|_m^2 & \lesssim \int_0^t  \sum_{y, z} P_{t-s}(x-y) \mathcal{R}(y-z) P_{t-s}(x-z) \left\|\sigma^{(N)} \big(U_s^{(N)}(y)\big)-\sigma\big(U_s^{(N)}(y)\big)\right\|_m^2\,\d s\\
&\lesssim\frac{1}{N^2} \int_0^t  \sum_{y, z} P_{t-s}(y) \mathcal{R}(y-z) P_{t-s}(z)\,\d s\\
&\lesssim\frac{1}{N^2}.
\end{align*}
Let
\bes
\mathscr{D}(t) := \sup_{y\in \Z^d}\big\|U_t^{(N)}(y) - U_t(y)\big\|_m^2.
\ees
We therefore obtain from equation \eqref{eq:t1t2}
\bes
\mathscr{D}(t) \le \frac{c_1}{N^2} + c_1\int_0^t \d s\,  \mathscr{D}(s) \,  \sum_{y,z}  P_{t-s}(y) \mathcal{R}(y-z) P_{t-s}(z),
\ees
for some constant $c_1$. We multiply each side of the above inequality by $e^{-\eta t}$ to obtain
\begin{align*}
e^{-\eta t}\mathscr{D}(t) &\le \frac{c_1e^{-\eta t}}{N^2} + c_1\int_0^t   e^{-\eta s}\mathscr{D}(s) \, e^{-\eta (t-s)} \sum_{y,z\in \Z^d}  P_{t-s}(y) \mathcal{R}(y-z) P_{t-s}(z)\,\d s\\
&\le\frac{c_1}{N^2}+c_1\left(\sup_{r>0}e^{-\eta r}\mathscr{D}(r)\right)\int_0^te^{-\eta s} \sum_{y,z}  P_s(y) \mathcal{R}(y-z) P_{s}(z)\,\d s.
\end{align*}
We now choose $\eta>0$ large enough so that 
\begin{align*}
c_1\int_0^te^{-\eta s} \sum_{y,z\in \Z^d}  P_s(y) \mathcal{R}(y-z) P_{s}(z)\,\d s\leq \frac{1}{2},
\end{align*}
and thus obtain
\begin{align*}
\sup_{t>0}e^{-\eta t}\mathscr{D}(t)\leq \frac{c_1}{N^2}+\frac{1}{2}\sup_{t>0}e^{-\eta t}\mathscr{D}(t),
\end{align*}
Some computations finish the proof of the proposition.
\end{proof}
Our second approximation result allows us to consider smooth $\sigma$. Let $\phi \in \mathbf{C}_c^\infty\big((0,1)\big)$ with $\int_{\R} \phi(x) \d x =1$ and by a slight abuse of notation, set
\bes
\sigma^{(N)}(x) := N \int_{\R} \d y \, \phi\big(N(y-x)\big) \sigma(y).
\ees
\begin{proposition} \label{prop:sigma:smoo} 
Consider \eqref{eq:disc:she} but with $\sigma^{(N)}$ instead of  $\sigma$ and let $U^{(N)}$ be its unique solution. Then $\lim_{N\to \infty} U^{(N)}_t(x)=U_t(x)$ in $L^m(\P)$ for every $m\ge 2$.
\end{proposition}
\begin{proof}
We adopt the same notation as in the proof of the previous result. We begin by making an observation which follows from the definition of $\sigma^{(N)}$.  We have
\begin{align*}
\sigma^{(N)}(x)-\sigma(x)&=N\int_\R\phi(N(y-x))(\sigma(y)-\sigma(x))\,\d y\\
&\lesssim\frac{1}{N},
\end{align*}
where we have used the definition of $\phi$ and the fact that $\sigma$ is Lipschitz. The proof is now exactly the same as that of Proposition \ref{prop:sigma:comp} except for the following where we make use of the above inequality,
\bes\begin{split}
\|\mathcal{T}_1\|_m^2  &\le c_2 \int_0^t \|\sigma^{(N)} \big(U_s^{(N)}(y)\big)-\sigma\big(U_s^{(N)}(y)\big)\|_m^2 \sum_{y, z} P_{t-s}(x-y) \mathcal{R}(y-z) P_{t-s}(x-z) \d s  \\
& \lesssim\frac{1}{N^2}\int_0^t \d s \, \sum_{y, z} P_{s}(y) \mathcal{R}(y-z) P_{s}(z) \lesssim \frac{1}{N^2}.
\end{split}\ees
The rest of the proof is similar to that of Proposition \ref{prop:sigma:comp}.
\end{proof}

Our final approximation concerns the state space. We consider a system of interacting SDEs on the space $K_N:=[-N,N]^d$. More precisely, let $U^{(N)}$ solve
\be\label{eq:disc:torus}
\d U^{(N)}_t(x) = \big(\mathscr{L}^{(N)}U^{(N)}_t\big)(x)\, \d t + \sigma\big(U^{(N)}_t(x)\big)\, \d B_t(x), \quad x\in K_N
\ee
where
\bes
\mathscr{L}^{(N)} = \nu \sum_{i,j\in K_N} \big(p^{(N)}_{i,j } - \delta_{i,j}\big)z_j \frac{\partial}{\partial z_i},\qquad p_{i,j}^{(N)} = \sum_{k \in \Z^d} p_{i, j+(2N+1)k}, \quad i, j \in K_N.
\ees
The correlations of the Brownian motions are as before \eqref{eq:bm:corr}. Note that $\mathscr{L}^{(N)}$ is the generator of a random walk $X^{(N)}$ which takes values on the discrete torus $[-N,N]^d$. We
shall denote the transition probability by $P^{(N)}_t(y):= P(X_t^{(N)}=-y),\, y \in K_N$.
\begin{proposition}\label{prop:i_set:fin} Let $U^{(N)}$ solve \eqref{eq:disc:torus}. Then $\lim_{N\to \infty} U_t^{(N)}(x)= U_t(x)$ in $L^m(\P)$ for every $m\ge 2$ and any $x\in \bZ^d$. 
\end{proposition}
\begin{proof} We get that $U_t^{(N)}(x)- U_t(x) = \mathcal{T}_1+\mathcal{T}_2+\mathcal{T}_3+\mathcal{T}_4$, where
\bes\begin{split}
\mathcal{T}_1 &= \big(P_t^{(N)}* U_0\big) (x) - \big(P_t* U_0\big)(x), \\
\mathcal{T}_2 &= \sum_{y\in K_N}\int_0^t P_{t-s}(x-y) \cdot \Big(\sigma\big(U^{(N)}_s(y)\big)-\sigma\big(U_s(y)\big)\Big)\,\d B_s(y) ,\\
\mathcal{T}_3&= \sum_{y\in K_N} \int_0^t \Big( P^{(N)}_{t-s}(x-y) -P_{t-s}(x-y)\Big) \cdot \sigma\big(U^{(N)}_s(y)\big) \, \d B_s(y) ,\\
\mathcal{T}_4 &=  -\sum_{y\in K_N^c} \int_0^t P_{t-s}(x-y) \cdot \sigma\big(U_s(y)\big) \, \d B_s(y).
\end{split}\ees
As before let us call
\bes
\mathscr{D}(t) :=\sup_{y\in \Z^d}\big\|U_t(y)-U_t^{(N)}(y)\big\|_m^2.
\ees
We get
\be\begin{split} \label{eq:dt:bd}
\mathscr{D}(t) &\le C\int_0^t \d s\, \mathscr{D}(s) \sum_{y,z\in K_N}P_{t-s}(y) \mathcal{R}(y-z) P_{t-s}(z) \\
&\quad  + \sup_{x\in \Z^d} \big\vert\big( P_t^{(N)}*U_0\big) (x) - \big(P_t*U_0\big) (x)\big\vert^2  + C \int_0^t  \d s\,\sum_{y,z\in K_N^c} P_{s}(y)\mathcal{R}(y-z) P_{s}(z) \\
&\quad + C \int_0^t \d s\, \sum_{y,z\in K_N} \big(P_s^{(N)}(y)-P_{s}(y)\big)\mathcal{R}(y-z) \big(P^{(N)}_{s}(z)-P_s(z)\big).
\end{split}\ee
The proof of Theorem \ref{thm:dis:mom:bd} shows that one can find a bound on $\|U_t^{(N)}(x)\|_m$ for $t\le T,\, x\in \Z^d$ which is independent of $N$. We have used this fact in the above the inequality. The second term goes to $0$ as $N \to \infty$ since the $P_t^{(N)}$ converges weakly to $P_t$. The third term converges to $0$
by the dominated convergence theorem. As for the last term, one bounds it by
 \bes \begin{split}
 & \int_0^t \d s \sum_{y, z \in K_N} \Big\lbrace \sum_{i\ne 0,\, i \in \Z^d} P_s\big(y +(2N+1) i\big)\Big\rbrace \cdot \mathcal{R}(y-z)  \cdot \Big\lbrace \sum_{j\ne 0,\,  j\in \Z^d} P_s\big(z+(2N+1) j\big) \Big \rbrace \\
& \le \mathcal{R}(0) \int_0^t \d s \;\P \big(X_s \text{ falls outside } K_N\big)^2,
\end{split} \ees
which tends to $0$ as $N$ tends to infinity. We thus have
\bes
\mathscr{D}(t) \le \mathcal{A}(N) + C\int_0^t \d s\, \mathscr{D}(s) \sum_{y,z\in \Z^d} P_{t-s}(y) \mathcal{R}(y-z) P_{t-s}(z)
\ees
for some function $\mathcal{A}(N)$  decreasing to $0$. By an argument similar to that used in the proof of
Proposition \ref{prop:sigma:comp} one concludes $\mathscr{D}(t) \to 0$. This completes the proof.
\end{proof}

Before we proceed we mention that a comparison result similar to Theorem \ref{thm:comp} holds for a system of interacting SDEs.  Indeed, Theorem 1.1 in \cite{geis-mant} implies that the comparison principle holds for finite dimensional
stochastic differential equations (SDEs) of the form \eqref{eq:disc:torus}. For the infinite dimensional case we can use the above proposition and the continuity of the solution. We need the comparison result to guarantee that the solution remains nonnegative 
provided that the initial profile is nonnegative and $\sigma(0)=0$. Let us now turn to the proof of the main result of this section.

\begin{proof}[Proof of Theorem \ref{thm:mom:disc:she}]The proof of the theorem follows the same strategy as in \cite{cox-flei-grev}. We therefore only mention the key points and leave it for the reader to check the details in \cite{cox-flei-grev}. We begin by proving the result under some simplifications. We assume that $\sigma$ is smooth and has support in $(a,\,b)$ with $0<a<b$. The system of SDEs is taken to be finite, that is we restrict $x\in K$ where $K$ is a finite set as defined in the above proposition. 

Let $S^{\sigma}$ denote  the strongly continuous contraction semigroup associated with the solution to the SDE with a particular diffusion coefficient $\sigma$. Also, let $G^\sigma$ be the corresponding generator given by
\bes
G^\sigma := \mathscr{L}  + \frac{1}{2}\sum_{i,j \in K} \sigma(z_i) \mathcal{R}(i-j) \sigma(z_j) \frac{\partial^2}{\partial z_i \partial z_j},\quad z\in [a,\,b]^{k},
\ees
where $k$ is the number of elements in $K$ and 
\begin{equation*}
\mathscr{L}:=\nu \sum_{i,j} (p_{i,j} - \delta_{i,j})z_j \frac{\partial}{\partial z_i}.
\end{equation*}
Consider the function $F_1$ of the form $F_1(z)=z_1^{n_1}z_2^{n_2}\cdots z_k^{n_k}$. As mentioned at the beginning of this section, we will show that 
\be \label{eq:comp:n=1}
S_{t_1}^{\sigma_1} F_1(z) \ge S_{t_1}^{\sigma_2} F_1(z),
\ee
by using the following formula,
\bes
S_{t_1}^{\sigma_1} F_1(z) - S_{t_2}^{\sigma_2} F_1(z) = \int_0^{t_1}S_{t_1-s}^{\sigma_2} \big(G^{\sigma_1}-G^{\sigma_2}\big) S_s^{\sigma_1}F_1(z)\,\d s.
\ees
That the right hand side of the above display is well defined follows from the proof of Lemma 15 of \cite{cox-flei-grev}. By a convexity argument as in the proofs of Propositions 16 and 17 of \cite{cox-flei-grev}, we have $\frac{\partial^2 }{\partial z_i \partial z_j} S_s^{\sigma_1}F_1 \ge 0$. Now since $\sigma_1\geq \sigma_2$ and 
\bes
 G^{\sigma_1}-G^{\sigma_2} =\frac{1}{2}\sum_{i,j \in K}  \mathcal{R}(i-j)\cdot \big[\sigma_1(z_i)\sigma_1(z_j)-\sigma_2(z_i)\sigma_2(z_j)\big] \frac{\partial^2}{\partial z_i \partial z_j}, 
\ees
we have 
\bes
\big(G^{\sigma_1}-G^{\sigma_2}\big) S_{s}^{\sigma_1} F_1(z)\ge 0,\quad 0\le s\le t_1.
\ees
We have thus proved \eqref{eq:comp:n=1}. Denote by $F_2(z)$ another function which is of the same form as $F_1(z)$. Using the Markov property we have for $t_1<t_2$
\begin{align*}
\E^{\sigma_1}_zF_1(U_{t_1}(\cdot))F_2(U_{t_2}(\cdot))&=
\E^{\sigma_1}_z\left[F_1(U_{t_1}(\cdot))\E^{\sigma_1}_{U_{t_1}(\cdot)}F_2(U_{t_2-t_1}(\cdot))\right]\\
&\geq \E^{\sigma_1}_z\left[F_1(U_{t_1}(\cdot))\E^{\sigma_2}_{U_{t_1}(\cdot)}F_2(U_{t_2-t_1}(\cdot))\right]\\
&\geq \E^{\sigma_2}_z\left[F_1(U_{t_1}(\cdot))\E^{\sigma_2}_{U_{t_1}(\cdot)}F_2(U_{t_2-t_1}(\cdot))\right]\\
&=\E^{\sigma_2}_z\left[F_1(U_{t_1}(\cdot))F_2(U_{t_2}(\cdot))\right],
\end{align*}
For the second last step we need \eqref{eq:comp:n=1} with $F_1$ replaced by
\bes
F^1(z)=F_1(z) \E_z^{\sigma_1}\left(F_2\big(U_{t_2-t_1}(\cdot)\big)\right).
\ees
See Proposition 17 in \cite{cox-flei-grev} for details. An induction argument shows 
\begin{align*}
\E^{\sigma_1}_zF_1(U_{t_1}(\cdot))F_2(U_{t_2}(\cdot))\cdots F_n(U_{t_n}(\cdot))\geq \E^{\sigma_2}_zF_1(U_{t_1}(\cdot))F_2(U_{t_2}(\cdot))\cdots F_n(U_{t_n}(\cdot))
\end{align*}
for any $n\geq 1$ with $F_i(z)$ chosen to have the same form as $F_1(z)$. This completes the proof under the simplifications which we remove by using Propositions \ref{prop:sigma:comp}, \ref{prop:sigma:smoo} and \ref{prop:i_set:fin}. See \cite{cox-flei-grev} for more details.
\end{proof}

%

\section{The fractional Laplacian and the approximations of the stable process} \label{sec:approx:stable}

In this section, we prove a {\it local limit theorem} which will be needed in the proof of Theorem \ref{thm:main}. We begin with a few important observations about $p_t(x)$, the heat kernel for the fractional Laplacian $-\nu (-\Delta)^{\alpha/2}$. 
\begin{itemize} 
\item  For any positive constant $c$, we have 
\begin{align}\label{scaling}
p_t(x)=c^dp_{c^\alpha t}(cx).
\end{align}
\item For $0<\alpha<2$, there exist positive constants $c_1$ and $c_2$ such that  
\begin{align}\label{heat-estimate}
c_1\left(\frac{1}{t^{d/\alpha}}\wedge\frac{t}{|x|^{d+\alpha}}\right)\leq p_t(x)\leq c_2\left(\frac{1}{t^{d/\alpha}}\wedge\frac{t}{|x|^{d+\alpha}}\right).
\end{align}
\end{itemize}
Note that the second property does not hold for $\alpha=2$. The scaling property follows from 
\begin{equation*}
p_t(x)=(2\pi)^{-d}\int_{\R^d}e^{-ix\cdot z}e^{-t\nu|z|^\alpha}\,\d z,
\end{equation*}
while the above two sided estimates are well known. See \cite{kolo} and references therein for various extensions. We will need the following straightforward consequence of \eqref{heat-estimate}. It also holds in the case $\alpha=2$.
\begin{lemma}\label{scale}
Fix $k \in \Z^d$. For all $z \in  C^{(\epsilon)}(\epsilon k)=\{x:=\epsilon k+y; -\epsilon/2\leq y_i< \epsilon/2 \}$ and $t\ge \epsilon^\alpha$, we have
\begin{equation*}
\frac{1}{c_1}p_t(z)\le p_t(\epsilon k)\leq c_1p_{t}(z),
\end{equation*}
where $c_1$ is some constant independent of $k$.
\end{lemma}

We will also need the following estimate. Results of this type are known \cite{kolo}; we give a simple proof using subordination.  Let $B_t$ be a $d$-dimensional Brownian motion and $T_t$ be a one-dimensional one-sided stable process of order $\alpha/2$    independent of $B_t$. Then, by subordination we have $Y_t=B_{T_t}$, where $Y_t$ is the strict stable process of order $\alpha$; see \cite{levy-davar}.
 
\begin{lemma}\label{ker-d}
For $x\in \R^d$ and $t>0$,
\begin{equation*}
\left |\nabla p_t(x) \right| \lesssim\frac{p_t(x/2)}{t^{1/\alpha}}.
\end{equation*}
\end{lemma}
\begin{proof}
Let $q_t(\cdot)$ denote the probability density function of $T_t$.  Then by subordination, we have 
\begin{align*}
p_t(x)=\int_0^\infty \frac{1}{(2\pi s)^{d/2}}e^{-\frac{|x|^2}{2s}}q_t(s)\,\d s.
\end{align*}
A simple computation shows that
\begin{align*}
\left|\nabla p_1(x)\right|&\lesssim\int_0^\infty \frac{1}{(2\pi s)^{d/2}}\frac{|x|}{s}e^{-\frac{|x|^2}{2s}}q_1(s)\,\d s\\
&\lesssim\int_0^\infty \frac{1}{s^{(d+1)/2}}e^{-\frac{|x|^2}{8s}}q_1(s)\,\d s\\
&\lesssim\left[\int_0^1 \frac{1}{s^{(d+1)/2}}e^{-\frac{|x|^2}{8s}}q_1(s)\,\d s +\int_1^\infty \frac{1}{s^{(d+1)/2}}e^{-\frac{|x|^2}{8s}}q_1(s)\,\d s\right].
\end{align*}
We estimate the second integral,
\begin{align*}
\int_1^\infty \frac{1}{s^{(d+1)/2}}e^{-\frac{|x|^2}{8s}}q_1(s)\,\d s&\leq \int_0^\infty \frac{1}{s^{d/2}}e^{-\frac{|x|^2}{8s}}q_1(s)\,\d s\\
&=c_1p_1(x/2).
\end{align*}
For the first integral, we have
\begin{align*}
\int_0^1 \frac{1}{s^{(d+1)/2}}e^{-\frac{|x|^2}{8s}}q_1(s)\,\d s&\leq e^{-\frac{|x|^2}{8}}\int_0^1 \frac{1}{s^{(d+1)/2}}q_1(s)\,\d s.
\end{align*}
By Theorem XIII.6.1 of \cite{feller}, as $ s\to 0$, $q_1(s)$ decays faster than any polynomial of $s$. Therefore the integral on the right is finite. Further for $\alpha<2$, the density $p_1(x)$ only decays polynomially in $x$ and therefore it is clear that for all $\alpha \le 2$ we have $ e^{-
|x|^2/8} \lesssim p_1(x/2)$. Combining the above inequalities gives $\left|\nabla p_1(x)\right|\lesssim p_1(x/2)$, from which we obtain the result by scaling.
\end{proof}

Our first local limit theorem gives a uniform (in $x$) bound on the difference between the scaled transition probabities
of the random walk and the heat kernel for the stable process. This is an improvement of Proposition 3.1 in \cite{jose-khos-muel}.
\begin{proposition} \label{prop:unif}
Fix $T>0$. Then under Assumption \ref{cond1},
\begin{equation*}
\sup_{x \in \epsilon \Z^d}\left|\frac{1}{\epsilon^d}P(\epsilon X_{t/\epsilon^\alpha}=x)-p_t(x)\right|\lesssim \frac{\epsilon^a}{t^{(d+a)/\alpha}}\quad\text{for}\quad x\in \epsilon \bZ^d 
\end{equation*}
uniformly for $\epsilon^\alpha\le t\le T$.
\end{proposition}

\begin{proof}
We begin by recalling that 
\begin{equation*}
p_t(x)=\frac{1}{(2\pi)^d}\int_{\R^d}e^{-ix\cdot z}e^{-t\nu|z|^\alpha}\,\d z,
\end{equation*}
and 
\begin{equation*}
P(X_t=x)=\frac{1}{(2\pi)^d}\int_{[-\pi,\,\pi]^d}e^{-ix\cdot z} e^{-t(1-\hat\mu(z))}\,\d z.
\end{equation*}
We therefore have 
\begin{align*}
(2\pi)^d\left|\frac{1}{\epsilon^d}P(\epsilon X_{t/\epsilon^\alpha}=x)-p_t(x)\right|&=\left|\frac{1}{\epsilon^d}\int_{[-\pi,\,\pi]^d}e^{-\frac{ix}{\epsilon}\cdot z} e^{-\frac{t}{\epsilon^\alpha}(1-\hat\mu(z))}\,\d z-\int_{\R^d}e^{-ix\cdot z}e^{-t\nu|z|^\alpha}\,\d z\right|\\
&\leq \left|\int_{[-\frac{\pi}{\epsilon},\,\frac{\pi}{\epsilon}]^d} e^{-ix\cdot z}\left[e^{-\frac{t}{\epsilon^\alpha}(1-\hat\mu(\epsilon z))}-e^{-t\nu|z|^\alpha}\right]\,\d z\right|\\
&\qquad+\left|\int_{\R^d \backslash [-\frac{\pi}{\epsilon},\,\frac{\pi}{\epsilon}]^d} e^{-ix\cdot z}e^{-t\nu|z|^\alpha}\,\d z\right|\\
&:=I_1+I_2.
\end{align*}
We bound $I_2$ first.
\begin{align*}
I_2&=\left|\int_{([-\frac{\pi}{\epsilon},\,\frac{\pi}{\epsilon}]^d)^c} e^{-ix\cdot z}e^{-t\nu|z|^\alpha}\,\d z\right|\\
&\leq e^{-\frac{c_1t\nu \pi^\alpha}{\epsilon^\alpha}}\int_{\R^d}e^{-\frac{t\nu|z|^\alpha}{2}}\,\d z\\
& \lesssim\frac{\epsilon^{\alpha N}}{t^N}\frac{1}{t^{d/\alpha}},
\end{align*}
where $N$ is some positive integer. Choosing $N$ to be the smallest integer bigger than $a/\alpha$ and using $t\geq \epsilon^\alpha$, the above inequality reduces to  
\begin{equation*}
I_2\lesssim\frac{\epsilon^a}{t^{(d+a)/\alpha}}.
\end{equation*}
Bounding $I_1$ is slightly harder. We begin by splitting the integral as follows.
\begin{align*}
I_1&\leq \left|\int_{A_{t\,\epsilon}} e^{-ix\cdot z}\left[e^{-\frac{t}{\epsilon^\alpha}(1-\hat\mu(\epsilon z))}-e^{-t\nu|z|^\alpha}\right]\,\d z\right|\\
&\qquad +\left|\int_{A_{t\,\epsilon}^c\cap [-\frac{\pi}{\epsilon},\,\frac{\pi}{\epsilon}]^d} e^{-ix\cdot z}\left[e^{-\frac{t}{\epsilon^\alpha}(1-\hat\mu(\epsilon z))}-e^{-t\nu|z|^\alpha}\right]\,\d z\right|\\
&=: I_3+I_4,
\end{align*}
where $$A_{t,\,\epsilon}:=\left\{z\in \R^d; |z|\leq \frac{1}{t^{1/(\alpha+a)}\epsilon^{a/(\alpha+a)}} \right \}.$$
Since we have $1-\hat\mu(\epsilon z)\geq c_2|\epsilon z|^\alpha$ for some positive constant $c_2$, we have 
\begin{align*}
\left|e^{-ix\cdot z}\left[e^{-\frac{t}{\epsilon^\alpha}(1-\hat\mu(\epsilon z))}-e^{-t\nu|z|^\alpha}\right]\right|\leq 2e^{-c_3t|z|^\alpha},
\end{align*}
with $c_3$ being another positive constant.
We now have 
\begin{align*}
I_4&\leq 2\int_{A_{t,\,\epsilon}^c\cap [-\frac{\pi}{\epsilon},\,\frac{\pi}{\epsilon}]^d}e^{-c_3t|z|^\alpha}\,\d z\\
&\lesssim \sup_{A_{t,\,\epsilon}^c}e^{-c_3t|z|^\alpha/2}\int_{\R^d}e^{-c_3t|z|^\alpha/2}\,\d z.
\end{align*}
On $A_{t,\,\epsilon}^c$, we have $|z|>\frac{1}{t^{1/(\alpha+a)}\epsilon^{a/(\alpha+a)}}$ and hence
\bes \begin{split}
\sup_{A_{t,\,\epsilon}^c}e^{-c_3t|z|^\alpha/2}&\leq e^{-\frac{c_3t}{2\epsilon^{a\alpha/(\alpha+a)}t^{\alpha/(\alpha+a)}}} \\
&\lesssim \frac{\epsilon^{ a\alpha\tilde{N}/(\alpha+a)}t^{\alpha \tilde{N}/(\alpha+a)}}{t^{\tilde{N}}}
\end{split}\ees
where $\tilde{N}$ is a positive integer. We therefore have 
\begin{align*}
I_4&\lesssim\frac{\epsilon^{ \alpha a\tilde{N}/(\alpha+a)}t^{\alpha \tilde{N}/(\alpha+a)}}{t^{\tilde{N}}}t^{d/\alpha}\\
&\lesssim\frac{\epsilon^{ \alpha a\tilde{N}/(\alpha+a)-a}t^{\alpha \tilde{N}/(\alpha+a)}}{t^{(\tilde{N}-a/\alpha)}}\frac{\epsilon^a}{t^{(d+a)/\alpha}}\\
&\lesssim \frac{\epsilon^a}{t^{(d+a)/\alpha}},
\end{align*}
where we have chosen $\tilde{N}$ to be large enough and used that $t\geq \epsilon^\alpha$. To bound $I_3$, we note that 
\begin{align*}
e^{-\frac{t}{\epsilon^\alpha}(1-\hat\mu(\epsilon z))}-e^{-t\nu|z|^\alpha}&=e^{-t\nu|z|^\alpha}[e^{\frac{t}{\epsilon^\alpha}\mathcal{D}(\epsilon z)}-1].
\end{align*}
Since we have $\frac{t}{\epsilon^\alpha}|\mathcal{D}(\epsilon z)|\leq \frac{c_4t}{\epsilon^\alpha}|\epsilon z|^{\alpha+a}$ for some constant $c_4$, on $A_{t,\,\epsilon}$, we have that $\frac{t}{\epsilon^\alpha}\mathcal{D}(\epsilon z)$ is bounded. We therefore have
\begin{align*}
I_3&\leq\int_{A_{t\,\epsilon}} e^{-t\nu|z|^\alpha}\left|e^{\frac{t}{\epsilon^\alpha}\mathcal{D}(\epsilon z)}-1\right| \d z\\
&\lesssim\int_{A_{t\,\epsilon}} e^{-t\nu|z|^\alpha}\frac{t}{\epsilon^\alpha}|\epsilon z|^{\alpha+a}\,\d z\\
&\le\int_{\R^d} e^{-t\nu|z|^\alpha}\frac{t}{\epsilon^\alpha}|\epsilon z|^{\alpha+a}\,\d z\\
&\lesssim\frac{\epsilon^a}{t^{(a+d)/\alpha}}.
\end{align*}
This completes the proof of the proposition.
\end{proof}
We next prove refinement of the above proposition. This will give us more information when $|x|\geq t^{1/\alpha}$. We need this because Riesz kernels have slowly decaying tails and we need to compensate for this by obtaining a better bound for large $x$. We need two lemmas whose proofs will be as useful as the results they describe.

\begin{lemma} Let $z \in \R^d\backslash\{0\}$. For any positive integer $k$ and real number $\gamma$,
\be \label{eq:abs_der}
\frac{\partial^k}{\partial z_i^k} |z|^\gamma = \sum_{n=1}^{n_k} A_n |z|^{\gamma-k} \Big(\frac{z_i}{|z|}\Big)^n,
\ee
where $A_n$'s are constants and $n_k$ denotes a positive integer depending on $k$. In particular, 
\be \label{eq:abs_der:bd}
\frac{\partial^k}{\partial z_i^k} |z|^\gamma \lesssim |z|^{\gamma-k}.
\ee
\end{lemma}
\begin{proof} We restrict to $i=1$. The proof follows by induction on $k$. The first derivative is
\bes
\frac{\partial}{\partial z_1} |z|^\gamma = \gamma |z|^{\gamma-1}\Big(\frac{z_1}{|z|}\Big).
\ees
Assume now that the \eqref{eq:abs_der} is true for some $k$. We use the product rule to differentiate \eqref{eq:abs_der} with respect to $z_1$ and obtain
 \bes
\frac{\partial^{k+1}}{\partial z_1^{k+1}} |z|^\gamma= \sum_{n=1}^{n_k} A_n (\gamma-k)|z|^{\gamma-k-1}\cdot \Big(\frac{z_1}{|z|}\Big)^{n+1} +\sum_{n=1}^{n_k} A_n |z|^{\gamma-k} \cdot n \Big(\frac{z_1}{|z|}\Big)^{n-1}\frac{1}{|z|}\Big[1- \Big(\frac{z_1}{|z|}\Big)^2\Big].
\ees
We now gather all the terms together to see
\bes
\frac{\partial^{k+1}}{\partial z_1^{k+1}} |z|^\gamma=\sum_{n=1}^{n_{k+1}} \tilde{A}_n |z|^{\gamma-{k+1}} \Big(\frac{z_1}{|z|}\Big)^n.
\ees
This is clearly of the same form as \eqref{eq:abs_der}. The second part of the lemma follows from the obvious bound $z_1/|z|\le 1$.
\end{proof}
An immediate consequence of the above result is the following.
\begin{lemma} \label{lem:exp:der:bd}
For any $z\in \R^d\backslash\{0\}$ and any positive integer $N$, we have 
\be \label{eq:exp:der:bd}
\left|\frac{\partial^{N}}{\partial z_i^{N}}e^{-\nu t|z|^\alpha}\right|\lesssim e^{-\nu t|z|^\alpha} \sum_{k=1}^N t^k|z|^{\alpha k-N}.
\ee
\end{lemma}
\begin{proof}
 Set $g(z):= e^{f(z)}$ for some smooth function $f$. Some calculus shows that
\be \label{eq:exp:der}
\frac{\partial^{N}}{\partial z_1^{N}}g(z)=e^{f(z)}\sum_{\substack{\mathbf{k}=(k_1, k_2,\cdots, k_N) \\ k_1+2k_2+\cdots dk_N=N}} A_{\mathbf{k}} \Big[\frac{\partial}{\partial z_1} f(z)\Big]^{k_1}\cdot \Big[\frac{\partial^2}{\partial z_1^2} f(z)\Big]^{k_2}\cdots \Big[\frac{\partial^N}{\partial z_1^N} f(z)\Big]^{k_N},
\ee
where the constants $A_{\mathbf{k}}$ depend on $N$. Using the above expression with $f(z)=-t\nu |z|^\alpha$ together with the previous lemma gives the result.
\end{proof}
The proof of the following result requires ideas from the theory of oscillatory integrals which deals with asymptotics of such integrals. The reader can learn about this from \cite{stei}.
\begin{proposition} \label{prop:x}
Fix $T>0$. Suppose that Assumptions \ref{cond1} and \ref{cond2} both hold. Then uniformly for  $\epsilon^\alpha\leq t\leq T$ and $|x|>t^{1/\alpha}, \, x\in \epsilon \Z^d$, we have
\begin{equation*}
\left|\frac{1}{\epsilon^d}P(\epsilon X_{t/\epsilon^\alpha}=x)-p_t(x)\right|\lesssim \frac{t\epsilon^a}{|x|^{d+\alpha+a}}.
\end{equation*}
\end{proposition}

\begin{proof}
Let $\phi:\R\rightarrow \R_+$ be a smooth, symmetric cutoff function with $\phi(z):=1$ for $|z|\leq 1$ and $\phi(z)=0$ for $|z|\geq 2$.
\begin{equation*}
\begin{aligned}
\frac{(2\pi)^d}{\epsilon^d}P(\epsilon X_{t/\epsilon^\alpha}=x)&=\frac{1}{\epsilon^d}\int_{[-\pi,\,\pi]^d}e^{-\frac{ix}{\epsilon}\cdot z} e^{-\frac{t}{\epsilon^\alpha}(1-\hat\mu(z))}\,\d z\\
&=\frac{1}{\epsilon^d}\int_{[-\pi,\,\pi]^d}e^{-\frac{ix}{\epsilon}\cdot z} e^{-\frac{t}{\epsilon^\alpha}(1-\hat\mu(z))}\phi(|z|t^{1/\alpha})\,\d z\\
& \qquad+\frac{1}{\epsilon^d}\int_{[-\pi,\,\pi]^d}e^{-\frac{ix}{\epsilon}\cdot z} e^{-\frac{t}{\epsilon^\alpha}(1-\hat\mu(z))}[1-\phi(|z|t^{1/\alpha})]\,\d z\\
&=I_1+I_2.
\end{aligned}
\end{equation*}
We first show that $I_2$ is small. Using integration by parts and the fact that $\hat\mu$ is periodic,
\begin{equation*}
\begin{aligned}
|I_2|\lesssim \frac{1}{\epsilon^d}\left(\frac{\epsilon}{|x_j|}\right)^N\left|\int_{[-\pi,\,\pi]^d}e^{-\frac{ix}{\epsilon}\cdot z} \frac{\partial^N}{\partial  z_j^N}\left\{e^{-\frac{t}{\epsilon^\alpha}(1-\hat\mu(z))}[1-\phi(|z|t^{1/\alpha})]\right\}\,\d z\right|,
\end{aligned}
\end{equation*}
where we shall choose later $N \le d+3$ and $0\leq j\leq d$. We next need an estimate on 
\begin{equation*}
\frac{\partial^N}{\partial  z_j^N}\left\{e^{-\frac{t}{\epsilon^\alpha}(1-\hat\mu(z))}[1-\phi(|z|t^{1/\alpha})]\right\}.
\end{equation*}
After some computations very similar to the proofs of the above two lemmas, we have by Assumptions \ref{cond1} and \ref{cond2} that for $k \le d+3$
\begin{equation}\label{diff1}
\left|\frac{\partial ^k}{\partial z_j^k}e^{-\frac{t}{\epsilon^\alpha}(1-\hat\mu(z))}\right|\lesssim\sum_{l=1}^k|z|^{l\alpha-k}\frac{t^l}{\epsilon^{l\alpha}}e^{-c_1\nu|z|^\alpha \frac{t}{\epsilon^\alpha}}.
\end{equation}
Note that  
\begin{equation}\label{diff2}
\left(\frac{|z|^\alpha t}{\epsilon^\alpha}\right)^{l}e^{-c_1\nu|z|^\alpha \frac{t}{\epsilon^\alpha}}\lesssim e^{-\frac{c_1\nu}{2}|z|^\alpha \frac{t}{\epsilon^\alpha}}.
\end{equation}
We now make a couple of observations.  For $1-\phi(|z|t^{1/\alpha})$ to be non-zero, $|z|$ has be be bigger than $t^{-1/\alpha}$ and derivatives of $1-\phi(|z|t^{1/\alpha})$ are non-zero only when $1\le |z|t^{1/\alpha}\le 2$. Moreover by an argument similar to \eqref{eq:exp:der:bd} 
\begin{equation} \label{eq:phi:der}
\left|\frac{\partial^k}{\partial z_j^k}\phi(|z|t^{\frac1\alpha})\right| \lesssim \frac{1}{|z|^k}\sum_{i=1}^kt^{\frac{i}{\alpha}}|z|^i.
\end{equation}
Using these observations along with Leibniz's rule gives
\begin{equation*}
\left|\frac{\partial ^N}{\partial z_j^N}\left\{e^{-\frac{t}{\epsilon^\alpha}(1-\hat\mu(z))}[1-\phi(|z|t^{1/\alpha})]\right\}\right|\leq c_2t^{N/\alpha}e^{-c_3|z|^\alpha \frac{t}{\epsilon^\alpha}}.
\end{equation*}
We therefore have 
\begin{equation*}
\begin{aligned}
\Big|\int_{[-\pi,\,\pi]^d}&e^{-\frac{ix}{\epsilon}\cdot z} \frac{\partial^N}{\partial  z_j^N}\left\{e^{-\frac{t}{\epsilon^\alpha}(1-\hat\mu(z))}[1-\phi(|z|t^{1/\alpha})]\right\}\,\d z \Big|\\
&\leq \int_{[-\pi,\,\pi]^d}\left|\frac{\partial^N}{\partial  z_j^N}\left\{e^{-\frac{t}{\epsilon^\alpha}(1-\hat\mu(z))}[1-\phi(|z|t^{1/\alpha})]\right\}\right|\,\d z\\
&\leq c_4\int_{[-\pi,\,\pi]^d}t^{N/\alpha}e^{-c_3|z|^\alpha \frac{t}{\epsilon^\alpha}}\,\d z\\
&\lesssim \frac{\epsilon^d}{t^{(d-N)/\alpha}} .
\end{aligned}
\end{equation*}
Putting the above estimates together and using that $j$ is arbitrary, we have
\begin{equation*}
I_2\lesssim\frac{\epsilon^{N}}{|x|^Nt^{(d-N)/\alpha}}.
\end{equation*}
Now choose $N$ an integer so that $N>d+\alpha+a$. Together with  $\epsilon^\alpha\le t \le T, \, |x|\ge t^{1/\alpha}$, we obtain 
\begin{equation*}
I_2\lesssim \frac{\epsilon^a t}{|x|^{d+\alpha+a}}.
\end{equation*}
We now look at $I_1$. We will decompose $I_1$ as follows.
\begin{equation*}
\begin{aligned}
I_1&=\frac{1}{\epsilon^d}\int_{[-\pi,\pi]^d}e^{-\frac{ix}{\epsilon}\cdot z} \left[e^{-\frac{t}{\epsilon^\alpha}(1-\hat\mu(z))}-e^{-\nu|z|^\alpha \frac{t}{\epsilon^\alpha}}\right]\phi(|z|t^{1/\alpha})\,\d z\\
&\qquad +\frac{1}{\epsilon^d}\int_{[-\pi,\pi]^d}e^{-\frac{ix}{\epsilon}\cdot z}e^{-\nu|z|^\alpha \frac{t}{\epsilon^\alpha}}\phi(|z|t^{1/\alpha}))\,\d z\\
&=I_3+I_4.
\end{aligned}
\end{equation*}
We look at $I_4$ first since it is the most straightforward part and bounding it involves the ideas used above. 
\begin{equation*}
\begin{aligned}
I_4&=\frac{1}{\epsilon^d}\int_{[-\pi,\pi]^d}e^{-\frac{ix}{\epsilon}\cdot z}e^{-\nu|z|^\alpha \frac{t}{\epsilon^\alpha}}\left[\phi(|z|t^{1/\alpha})-1\right]\,\d z
\\
&\qquad +\frac{1}{\epsilon^d}\int_{[-\pi,\pi]^d}e^{-\frac{ix}{\epsilon}\cdot z}e^{-\nu|z|^\alpha \frac{t}{\epsilon^\alpha}}\,\d z \\
& = \frac{1}{\epsilon^d}\int_{[-\pi,\pi]^d}e^{-\frac{ix}{\epsilon}\cdot z}e^{-\nu|z|^\alpha \frac{t}{\epsilon^\alpha}}\left[\phi(|z|t^{1/\alpha})-1\right]\,\d z \\
&\qquad+ (2\pi)^d p_t(x) -\frac{1}{\epsilon^d}\int_{\R^d\backslash [-\pi,\pi]^d}e^{-\frac{ix}{\epsilon}\cdot z}e^{-\nu|z|^\alpha \frac{t}{\epsilon^\alpha}}\,\d z \\
&= (2\pi)^d p_t(x)+O\left( \frac{ \epsilon^a t}{|x|^{d+\alpha+a}}\right).
\end{aligned}
\end{equation*}
The last line requires some explanation. The first term in the second last line can be bounded just as we 
did for $I_2$. For $N>d+\alpha+a$ we have from \eqref{eq:exp:der:bd} and \eqref{diff2}, and an appropriate choice of $j$,
\begin{equation*}\begin{split}
\left | \frac{1}{\epsilon^d}\int_{\R^d\backslash [-\pi,\pi]^d}e^{-\frac{ix}{\epsilon}\cdot z}e^{-\nu|z|^\alpha \frac{t}{\epsilon^\alpha}}\,\d z\right| & \lesssim \frac{1}{\epsilon^d} \left(\frac{\epsilon}{|x_j|}\right)^N \int_{\R^d\backslash[-\pi,\pi]^d} \left| \frac{\partial^N}{\partial z_1^N} e^{-\nu |z|^\alpha\frac{t}{\epsilon^\alpha}}\right| \, \d z \\
&  \lesssim \frac{1}{\epsilon^d} \left(\frac{\epsilon}{|x_j|}\right)^Ne^{-c_5 \frac{t}{\epsilon^\alpha}}  \\
&\lesssim \frac{ \epsilon^a t}{|x|^{d+\alpha+a}},
\end{split}
\end{equation*}
since $|x|^\alpha \ge t \ge \epsilon^\alpha$. We next turn our attention to $I_3$. We split the integral as follows.
\begin{align*}
I_3&=\frac{1}{\epsilon^d}\int_{[-\pi,\pi]^d}e^{-\frac{ix}{\epsilon}\cdot z} \left[e^{-\frac{t}{\epsilon^\alpha}(1-\hat\mu(z))}-e^{-\nu|z|^\alpha \frac{t}{\epsilon^\alpha}}\right]\phi(|z|t^{1/\alpha})\,\d z\\
&=\frac{1}{\epsilon^d}\int_{[-\pi,\pi]^d}e^{-\frac{ix}{\epsilon}\cdot z} \left[e^{-\frac{t}{\epsilon^\alpha}(1-\hat\mu(z))}-e^{-\nu|z|^\alpha \frac{t}{\epsilon^\alpha}}\right]\phi(|z|t^{1/\alpha})\psi(|z|)\,\d z\\
&\qquad+\frac{1}{\epsilon^d}\int_{[-\pi,\pi]^d}e^{-\frac{ix}{\epsilon}\cdot z} \left[e^{-\frac{t}{\epsilon^\alpha}(1-\hat\mu(z))}-e^{-\nu|z|^\alpha \frac{t}{\epsilon^\alpha}}\right]\phi(|z|t^{1/\alpha})[1-\psi(|z|)]\,\d z\\
&=I_5+I_6,
\end{align*}
where $\psi(\cdot)$ is a radial and smooth nonnegative function which equals to $1$ inside a ball of radius $\lambda \le 1$ (to be chosen later) and zero outside a ball of size $2\lambda$.  For $|z|\le 2$, we have 
\begin{align*}
e^{-\frac{t}{\epsilon^\alpha}(1-\hat\mu(z))}&-e^{-\nu|z|^\alpha \frac{t}{\epsilon^\alpha}}\\
 &=\frac{t}{\epsilon^\alpha}\mathcal{D}(z)e^{-\frac{t}{\epsilon^\alpha}(\nu|z|^\alpha-\theta\mathcal{D}(z))}\quad\text{for}\quad 0<\theta<1,
\end{align*}
which we use to obtain the following,
\begin{align*}
I_5&\lesssim\frac{1}{\epsilon^d}\int_{[-\pi,\pi]^d} \left|e^{-\frac{t}{\epsilon^\alpha}(1-\hat\mu(z))}-e^{-\nu|z|^\alpha \frac{t}{\epsilon^\alpha}}\right|\phi(|z|t^{1/\alpha})\psi(|z|)\,\d z\\
&\lesssim\frac{1}{\epsilon^d}\int_{|z|\leq2 \lambda}\frac{t}{\epsilon^\alpha} \mathcal{D}(z)\,\d z\\
&\lesssim\frac{t\lambda^{a+\alpha+d}}{\epsilon^{\alpha+d}}.
\end{align*}
To bound $I_6$, we note that as before,
\begin{align*}
I_6\lesssim\left(\frac{\epsilon}{|x_j|}\right)^N\frac{1}{\epsilon^d}\int_{[-\pi,\pi]^d}\left|\frac{\partial ^N}{\partial z_j^N}\left\{e^{-\nu|z|^\alpha \frac{t}{\epsilon^\alpha}} \left[e^{-\frac{t}{\epsilon^\alpha}\mathcal{D}(z)}-1\right]\phi(|z|t^{1/\alpha})[1-\psi(|z|)]\right\}\,\right|\d z.
\end{align*}
Due to the presence $\psi$ in the integrand, it is nonzero only when $\lambda\le |z| \le C\pi$. In this region, Leibniz's rule gives
\begin{equation*}\begin{split}
&\left|\frac{\partial ^N}{\partial z_j^N}\left\{e^{-\nu|z|^\alpha \frac{t}{\epsilon^\alpha}}\phi(|z|t^{1/\alpha})\cdot \big[e^{-\frac{t}{\epsilon^\alpha}\mathcal{D}(z)}-1\big][1-\psi(|z|)]\right\} \right|\\
&\lesssim \sum_{k=0}^N\left|\frac{\partial ^k}{\partial z_j^k}\left\{e^{-\nu|z|^\alpha \frac{t}{\epsilon^\alpha}}\phi(|z|t^{1/\alpha}) \right\}\right| \cdot \left|\frac{\partial ^{N-k}}{\partial z_j^{N-k}} \left\{\big[e^{-\frac{t}{\epsilon^\alpha}\mathcal{D}(z)}-1\big][1-\psi(|z|)]\right\}\right| \\
& \lesssim \sum_{k=0}^N \frac{e^{-\nu|z|^\alpha \frac{t}{\epsilon^\alpha}}}{|z|^k} \cdot \frac{1}{|z|^{N-k-1}} \\
&\lesssim \frac{e^{-\nu|z|^\alpha \frac{t}{\epsilon^\alpha}}}{|z|^{N-1}},
\end{split}\end{equation*}
where we used \eqref{diff1}, \eqref{diff2} and \eqref{eq:phi:der} to control the first factor in the third last line,
and for the second factor the relations
\begin{equation*}
\frac{\partial^l}{\partial z_j^l} \psi(|z|) \lesssim \frac{1}{|z|^{l}} \sum_{i=1}^l |z|^{i},
\end{equation*}
and
\begin{equation*} \begin{split}
\left |\frac{\partial^l}{\partial z_j^l} \left[e^{-\frac{t}{\epsilon^\alpha}\mathcal{D}(z)}-1\right] \right|  & \lesssim e^{-\frac{t}{\epsilon^\alpha} \mathcal{D}(z)} \sum_{i=1}^l \left(\frac{t}{\epsilon^\alpha}\right)^i |z|^{(\alpha+a)i-l} \\
&\lesssim \frac{e^{-c_6\frac{t}{\epsilon^\alpha} \mathcal{D}(z)}}{|z|^l},
\end{split}\end{equation*}
which holds because of Assumption \ref{cond2} provided of course $l\le d+3$. After some computations, we obtain for $N=d+2$
\begin{align*}
I_6&\lesssim \left(\frac{\epsilon}{|x|}\right)^N\frac{1}{\epsilon^d}\int_{|z|\geq \lambda}\frac{e^{-\nu|z|^\alpha \frac{t}{\epsilon^\alpha}}}{|z|^{N-1}}\,\d z\\
&\lesssim \frac{\epsilon}{|x|^N}  \left(\frac{\lambda}{\epsilon}\right)^{d-N+1}.
\end{align*}
Choosing $\lambda=\frac{\epsilon}{|x|}$, we have 
\begin{align*}
I_3&\lesssim \frac{t\epsilon^a}{|x|^{d+\alpha+a}}.
\end{align*}
Combining all the above estimates yield the result.
\end{proof}

The proof of the following local limit theorem is now almost complete.
\begin{theorem}\label{local-limit}
Let $\epsilon>0$. 
If Assumption \ref{cond1} holds then uniformly for $x\in \epsilon \bZ^d,\, \epsilon^\alpha\le t\le T$,
\begin{equation*}
\left|\frac{1}{\epsilon^d}P(\epsilon X_{t/\epsilon^\alpha}=x)-p_t(x)\right|\lesssim \frac{\epsilon^a}{t^{(d+a)/\alpha}}.
\end{equation*}
If both Assumptions \ref{cond1} and \ref{cond2} hold then uniformly for $\epsilon^\alpha\le t\le T$ and $|x|\ge t^{1/\alpha},\, x\in \epsilon \Z^d$, 
\begin{equation*}
\left|\frac{1}{\epsilon^d}P(\epsilon X_{t/\epsilon^\alpha}=x)-p_t(x)\right|\lesssim \frac{t\epsilon^a}{|x|^{d+\alpha+a}}.
\end{equation*}
For $0<\alpha<2$, the above inequality reduces to
\begin{equation}\label{eq:combined}
\left|\frac{1}{\epsilon^d}P(\epsilon X_{t/\epsilon^\alpha}=x)-p_t(x)\right|\lesssim \frac{\epsilon^ap_t(x)}{t^{a/\alpha}},
\end{equation}
uniformly for all $x\in \epsilon \Z^d$ and $\epsilon^\alpha \le t\le T$.
\end{theorem}
\begin{proof}
We only need to justify the final inequality. But this follows easily from \eqref{heat-estimate} and the first two inequalities in the statement of the theorem.
\end{proof}
\begin{remark} Note that \eqref{eq:combined} is not true for $\alpha=2$.
\end{remark}
We end this section by describing a class of random walks whose characteristic functions satisfy both Assumptions \ref{cond1} and \ref{cond2}.
\begin{example}[$\alpha=2$]\normalfont \label{eg1} Consider a $d$ dimensional random walk with dislocation distribution 
given by $\mu = \mu_1\otimes \mu_2 \otimes \cdots \otimes \mu_d$ where $\mu_k$ are the dislocation distributions of 
one dimensional mean $0$ random walks. Further assume that each of the one dimensional walks has moments of order $d+4$. This
implies that the characteristic functions of each of the $\hat{\mu}_k$ are continuously differentiable up to order $d+4$, and have a Taylor expansion of the form $\hat\mu_k(z_k)=1-c_2 z_k^2 + c_3(iz_k)^3+\cdots + c_{d+3} (i z_k)^{d+3} +o(|z_k|^{d+4}).$ One can check that these random walks satisfy Assumptions \ref{cond1} and \ref{cond2} for any $0<a<1$.
\end{example}
\begin{example}[$0<\alpha<2$] \normalfont \label{eg2}
Consider
\begin{align*} 
\mu(\{j\})=\frac{c_1}{|j|^{d+\alpha}}\quad\text{whenever}\quad j\in \Z^d/\{0\}.
\end{align*}
where the constant $c_1$ is chosen so that the above is a probability measure. In this case Assumptions \ref{cond1} and \ref{cond2} hold with $a\le 2-\alpha$. In fact \eqref{eq:cond2} holds for all $k\ge 1$. For the reader's convenience, we present the argument below.
For $|z| \le \pi$
\begin{equation}\label{eq:ex:bd}
\left \vert \sum_{j \neq 0} \frac{1-\cos(z\cdot j)}{|j|^{d+\alpha}} - \int_{\R^d}\frac{1-\cos(z\cdot x)}{|x|^{d+\alpha}} \,\d x \right \vert 
\lesssim \begin{cases} |z|^{1+\alpha}, &0<\alpha<1, \\ |z|^2 \ln (|z|^{-1}), &  \alpha =1, \\ |z|^2, & 1<\alpha<2.\end{cases}
\end{equation}
We now give more details of this calculation. We split the integral over $\R^d$ into blocks of the
form $[j-\frac12, j+\frac12]^d,\, j \in \Z^d$. We have the bound
\bes
\left \vert \int_{[-\frac12,\frac12]^d }\frac{1-\cos(z\cdot x)}{|x|^{d+\alpha}} \,\d x \right \vert  \lesssim |z|^2,
\ees
when one uses the inequality $\vert 1-\cos(z\cdot x) \vert \le  |z|^2|x|^2$. We next bound 
\bes 
\sum_{j \ne 0} \int_{[j-\frac12,\,j+\frac12]^d} \left \vert \frac{1-\cos(z\cdot x)}{|x|^{d+\alpha}}-\frac{1-\cos(z\cdot j)}{|j|^{d+\alpha}} \right \vert \, \d x.
\ees
For this we use a first order Taylor approximation around each $j$. We thus need a bound on $|| \partial f/\partial x_i||_\infty$ and $||\partial^2 f/ \partial x_i \partial x_j ||_\infty$  in $[j-\frac12\,,j+\frac12]^d$ where
\bes
f(x) = \frac{1-\cos(z\cdot x)}{|x|^{d+\alpha}}.
\ees
One can check that 
\bes \begin{split}
\frac{\partial f}{\partial x_i } &= \frac{z_i\sin(z\cdot x)}{|x|^{d+\alpha}} -\frac{\{1-\cos(z\cdot x)\}(d+\alpha) x_i}{|x|^{d+\alpha+2}}, \\
\frac{\partial^2 f}{\partial x_i \partial x_j}&= -(d+\alpha)\frac{(z_ix_j+z_jx_i) \cdot \sin(z\cdot x)+\ind\{i=j\}\cdot[1-\cos(z\cdot x)]}{|x|^{d+\alpha+2}} \\
&\qquad \qquad -(d+\alpha)(d+\alpha+2) \frac{[1-\cos(z\cdot x)]x_ix_j}{|x|^{d+\alpha+4}}+\frac{z_i z_j \cos(z\cdot x)}{|x|^{d+\alpha}}.
\end{split}\ees
We therefore get 
\bes
\begin{split}
\sum_{j \ne 0} \int_{[j-\frac12,\,j+\frac12]^d} \left \vert\frac{\partial f}{\partial x_i } \right \vert  \d x &\le  \int_{|x|>\frac12} \left \vert \frac{z_i\sin(z\cdot x)}{|x|^{d+\alpha}}\right \vert+ \left \vert\frac{\{1-\cos(z\cdot x)\}(d+\alpha) x_i}{|x|^{d+\alpha+2}}\right \vert\, \d x \\
& \lesssim |z|^2 \int_{\frac12}^{|z|^{-1}} \frac{\d r}{r^{\alpha}} + |z|\int_{|z|^{-1}}^{\infty}\frac{\d r}{r^{1+\alpha}}  + \int_{|z|^{-1}}^{\infty}\frac{\d r}{r^{2+\alpha}},
\end{split}
\ees
where we use the bounds $|\sin(z\cdot x)| \le \min(1,\,|z||x|)$ and $|1-\cos(z\cdot x)| \le \min(1,\,|z|^2|x|^2)$. The integral of the second derivatives can be bound in a similar manner. This gives us \eqref{eq:ex:bd}.

Consider now the $C^\infty$ function 
\bes
g(x) = \frac{c_1\cdot\eta(x)}{|x|^{d+\alpha}}
\ees
where $\eta$ is a $C^\infty$ function with $\eta(x)=0$ for $|x|\le \frac14$ and $\eta(x)=1$ for $|x|\ge \frac12$.
The Fourier transform of $g$ is
\bes
\mathcal{F}g(\xi) = \int_{\R^d} g(x) e^{ i x\cdot \xi} \, \d x.
\ees
We now split
\be \label{eq:g:split}
g(x)= \frac{c_1}{|x|^{d+\alpha}} + \frac{c_1\cdot(\eta(x)-1)}{|x|^{d+\alpha}}
\ee
The Fourier transform of the first term (in the sense of distributions) is 
\be \label{eq:fourier:riesz}
\mathcal{F}\big(c_1|\cdot|^{-d-\alpha}\big) \, (\xi) = \nu |\xi|^{\alpha}
\ee
for some constant $\nu$. For a proof see pages 127-128 of \cite{graf-08}.  The Fourier transform of the second term on the right of \eqref{eq:g:split} is $C^\infty$
since it is the Fourier transform of a compact distribution. Thus $\mathcal F g$ is equal to the sum of $c_2|\xi|^\alpha$ and a $C^\infty$ function, and by an integration by parts argument one can show that $\mathcal F g$ is a rapidly decaying function. We can now use Poisson summation formula to conclude that for $|z| \le \pi$
\bes
c_1\sum_{j \ne 0} \frac{ \cos(z\cdot j)}{|j|^{d+\alpha}} = \nu|z|^\alpha +h(z)
\ees
where $h$ is a $C^\infty$ function, and thus 
\bes
1-\hat\mu(z) = - \nu |z|^\alpha+\big(1-h(z)\big).
\ees
It is an easy computation to show that 
\bes
\int_{\R^d} \frac{1-\cos(z\cdot x)}{|x|^{d+\alpha}}\, \d x = c_2|z|^\alpha.
\ees
Therefore by \eqref{eq:ex:bd} we can conclude that $\nu=-c_2$ and 
\bes
| 1- h(z) | \lesssim \begin{cases} |z|^{1+\alpha}, &0<\alpha<1, \\ |z|^2 \ln (|z|^{-1}), &  \alpha =1, \\ |z|^2, & 1<\alpha<2.\end{cases}
\ees
Thus $h(0)=1$ and $\partial h/\partial x_i(0)=0$ for all $i$. Take $\mathcal{D}(z)=1-h(z)$, which is a $C^\infty$
function whose Taylor expansion around $0$ has the zeroth and first order terms to be $0$. It thus satisfies \eqref{eq:muhat:zero} and \eqref{eq:cond2} with $a\le 2-\alpha$. We have therefore found a random walk, $X_t$ satisfying the required conditions for a fixed $\nu$ satisfying \eqref{eq:fourier:riesz}. For other values of $\nu$. we can simply take the same random walk but with a different time scale.

The argument above shows that walks of the following form would also satisfy Assumptions \ref{cond1} and \ref{cond2}:
\bes
\mu(\{j\}) = \sum_{i=1}^{m} \frac{c_i}{|j|^{d+\alpha_i}}, \quad j \ne 0,
\ees
for some $m\in \Z_+$, $ \alpha=\alpha_1<\alpha_2<\cdots< \alpha_m<2$ and appropriate positive constants $c_1,c_2,\cdots, c_m$.
\end{example}
\section{Proof of Theorem \ref{thm:main}} \label{sec:thm:main}

We begin this section with the following result.

\begin{proposition}\label{prop-cor}
For $x,\,y\in \R^d$, we have 
\begin{equation*}
\int_{\R^d}\int_{\R^d}p_t(x-w)p_t(y-z)f(w-z)\,\d w\d z\lesssim \frac{1}{t^{\beta/\alpha}}.
\end{equation*}

\end{proposition}
\begin{proof}
By the semigroup property, we have 
\begin{align*}
\int_{\R^d}\int_{\R^d}p_t(x-w)&p_t(y-z)f(w-z)\,\d w\d z\\
&= \int_{\R^d} p_{2t}(w-(x-y)) f(w) \d w.
\end{align*}
The result now follows by a change of variable and scaling properties.
\end{proof}

Define
\be\begin{split} \label{eq:eta:eta'}
\eta := \frac{\alpha-\beta}{2} \qquad\text{and} \qquad\tilde{\eta}:= \frac{\alpha-\beta}{2\alpha}.
\end{split}\ee

We have the following H\"older continuity estimate. This can be read from \cite{boul-lahc-edda}. 
\begin{proposition}
For any  $m\geq 2$, we have 
\begin{align*}
\E|u_s(x)-u_t(y)|^m\lesssim |x-y|^{\eta m}+|s-t|^{\tilde{\eta}m}.
\end{align*}
\end{proposition}

The proof of Theorem \ref{thm:main} will involve several approximations which we will analyse in the following Lemmas. For the sake of clarity, they will be proved under the assumption that $u_0\equiv1$. We will also use the fact that $\sup_x\E|u_t(x)|^m$ is uniformly bounded for $0<t<T$.  

The first approximation is a step function approximation to $u$ which is constant over rectangles $C^{(\epsilon)}(\epsilon k), \, k \in\Z^d$. For $x\in \R^d$, set
\bes
u_t^{(\epsilon)}(\epsilon[x/\epsilon]) := 1+\sum_{k\in \Z^d} \int_0^{t-\epsilon^{\alpha}}\int_{C^{(\epsilon)}(\epsilon k)} p_{t-s}\left(y-\epsilon[x/\epsilon]\right) \sigma\big(u_s(\epsilon k)\big)\, F(\d s \, \d y),
\ees
if $t \ge \epsilon^\alpha$ and $0$ otherwise. If we set $\gamma(y):=\epsilon k$ when $y\in C^{(\epsilon)}(\epsilon k)$, then the above simplifies to
\bes
u_t^{(\epsilon)}(\epsilon[x/\epsilon]) := 1+ \int_0^{t-\epsilon^{\alpha}}\int_{\R^d} p_{t-s}(y-\epsilon[x/\epsilon]) \sigma\big(u_s(\gamma(y))\big)\, F(\d s \, \d y).
\ees
It is intuitively clear that $u^{(\epsilon)}$ should be close to $u$. The following lemma makes it precise.

\begin{lemma}
For $t>0$ and $x\in \R^d$, we have 
\begin{equation*}
\E \left|u_t(\epsilon[x/\epsilon])-u_t^{(\epsilon)}(\epsilon[x/\epsilon])\right|^m\lesssim \epsilon^{\eta m}.
\end{equation*}
\end{lemma}

\begin{proof}
Using the mild formulation of the solution and the above definition, we have
\begin{align*}
u_t(\epsilon[x/\epsilon])-u_t^{(\epsilon)}(\epsilon[x/\epsilon])&= \int_0^{t-\epsilon^{\alpha}}\int_{\R^d} p_{t-s}(y-\epsilon[x/\epsilon]) [\sigma\big(u_s(y)\big)-\sigma\big(u_s(\gamma(y))\big)]\, F(\d s \, \d y)\\
&\quad +\int_{t-\epsilon^\alpha}^t\int_{\R^d}p_{t-s}(y-\epsilon[x/\epsilon]) \sigma\big(u_s(y)\big)\, F(\d s \, \d y).
\end{align*}
An application of Burkholder's inequality together with the assumption on $\sigma$ yield
\begin{align*}
&\left\|u_t(\epsilon[x/\epsilon])-u_t^{(\epsilon)}(\epsilon[x/\epsilon])\right\|_m^2\\
&\lesssim\int_0^{t-\epsilon^{\alpha}}\int_{\R^d\times \R^d} p_{t-s}(y-\epsilon[x/\epsilon])p_{t-s}(w-\epsilon[x/\epsilon])f(y-w) \cC_{s}(y,\,w)\, \d s \, \d y\, \d w\\
&\;+\int_{t-\epsilon^\alpha}^t\int_{\R^d\times \R^d}p_{t-s}(y-\epsilon[x/\epsilon])p_{t-s}(w-\epsilon[x/\epsilon]) f(y-w)\left\|\sigma\big(u_s(y)\big)\sigma\big(u_s(w)\big)\right\|_{m/2}\, \d s \, \d y\,\d w\\
& =I_1+I_2,
\end{align*}
where $\cC_{s}(y,\,w):=\left\|\left[\sigma\big(u_s(y)\big)-\sigma\big(u_s(\gamma (y))\big)\right] \cdot \left[\sigma\big(u_s(w)\big)-\sigma\big(u_s(\gamma(w))\big)\right]\right\|_{m/2}.$ An application of the previous proposition yields $\cC_{s}(y,\,w)\leq c_1\epsilon^{2\eta}$ which we use in the following,  
\begin{align*}
I_1&\lesssim \epsilon^{2\eta }\int_0^{t-\epsilon^{\alpha}}\int_{\R^d\times \R^d} p_{t-s}(y-\epsilon[x/\epsilon])p_{t-s}(w-\epsilon[x/\epsilon])f(y-w)\,\d s \, \d y\, \d w.
\end{align*}
By translation invariance the right hand side does not depend on $\epsilon[x/\epsilon]$, and so
\begin{align*}
I_1\lesssim \epsilon^{2\eta }.
\end{align*}
We use the fact that solution has finite moments to bound $I_2$; 
\begin{align*}
I_2
&\lesssim \int_{t-\epsilon^\alpha}^t\int_{\R^d\times \R^d}p_{t-s}(y-\epsilon[x/\epsilon])p_{t-s}(w-\epsilon[x/\epsilon])f(y-w)\,\d s \, \d y\,\d w\\
&\lesssim \epsilon^{2\eta }.
\end{align*}
The proof is complete.
\end{proof}
We now turn to the second lemma. Here we discretize the density of the stable process. We set 
\bes
v_t^{(\epsilon)}(\epsilon[x/\epsilon]) := 1+\sum_{k\in \Z^d} \int_0^{t-\epsilon^{\alpha}}\int_{C^{\epsilon}(\epsilon k)} p_{t-s}(\epsilon k-\epsilon[x/\epsilon]) \sigma\big(u_s(\epsilon k)\big)\, F(\d s \, \d y),
\ees
for $t\ge \epsilon^\alpha$ and $0$ otherwise.

\begin{lemma}For any $t>0$ and $x\in \R^d$,
\begin{equation*}  
\E\left |u_t^{(\epsilon)}(\epsilon[x/\epsilon])-v_t^{(\epsilon)}(\epsilon[x/\epsilon])\right|^m \lesssim \epsilon^{\eta m}.
\end{equation*}
\end{lemma}

\begin{proof}
We obviously have 
\begin{align*}
u_t^{(\epsilon)}(\epsilon[x/\epsilon])&-v_t^{(\epsilon)}(\epsilon[x/\epsilon])\\
&=\int_0^{t-\epsilon^{\alpha}}\int_{\R^d} \left[p_{t-s}(y-\epsilon[x/\epsilon])-p_{t-s}(\gamma(y)-\epsilon[x/\epsilon])\right] \sigma\big(u_s(\gamma(y))\big)\, F(\d s \, \d y).
\end{align*} 
As in the previous lemma, the following holds,
\begin{align*}
&\left\|u_t^{(\epsilon)}(\epsilon[x/\epsilon])-v_t^{(\epsilon)}(\epsilon[x/\epsilon])\right\|_m^2\\
&\lesssim\int_0^{t-\epsilon^{\alpha}}\int_{\R^d\times \R^d}\Big [p_{t-s}(y-\epsilon[x/\epsilon])-p_{t-s}(\gamma(y)-\epsilon[x/\epsilon])\Big]\Big[p_{t-s}(w-\epsilon[x/\epsilon])\\
&\hspace{2cm}-p_{t-s}(\gamma(w)-\epsilon[x/\epsilon]) \Big]\cdot \left\|
\sigma\big(u_s(\gamma(y))\big)\sigma\big(u_s(\gamma(w))\big)\right\|_{m/2}f(y-w)\, \d s \, \d y\,\d w\\
&\lesssim \int_0^{t-\epsilon^{\alpha}}\int_{\R^d\times \R^d} \Big[p_{t-s}(y)-p_{t-s}(\gamma(y))\Big]\Big[p_{t-s}(w)-p_{t-s}(\gamma(w))\Big]f(y-w)\,\d s\,\d y\,\d w.
\end{align*}
By the mean value theorem, Lemma \ref{ker-d}, Lemma \ref{scale} and Proposition \ref{prop-cor}, we have 
\begin{align*}
\int_{\R^d\times \R^d} &\big[p_{t-s}(y)-p_{t-s}(\gamma(y))\big]\big[p_{t-s}(w)-p_{t-s}(\gamma(w))\big]f(y-w)\,\d s\,\d y\,\d w\\
&\lesssim \frac{\epsilon^2}{(t-s)^{2/\alpha}}\int_{\R^d\times \R^d} p_{t-s}(y/2)p_{t-s}(w/2)f(y-w)\,\d s\,\d y\,\d w \\
&\lesssim \frac{\epsilon^2}{(t-s)^{\frac{2}{\alpha}}(t-s)^{\frac{\beta}{\alpha}}}.
\end{align*}
 The rest of the proof is elementary calculus.
\end{proof}
The next proposition is crucial in that it determines the rate of convergence in Theorem \ref{thm:main}. Here we replace the discretized density by the transition probabilities for the random walk. Set 
\bes
V_t^{(\epsilon)}(\epsilon[x/\epsilon]) := 1+\sum_{k\in \Z^d} \int_0^{t-\epsilon^{\alpha}} \int_{C^{(\epsilon)}(\epsilon k)} \frac{P^{(\epsilon)}_{t-s}(\epsilon k-\epsilon[x/\epsilon])}{\epsilon^d} \sigma\big(u_s(\epsilon k)\big)\, F(\d s \, \d y),
\ees
for $t\ge \epsilon^\alpha$ and $0$ otherwise, where
\begin{equation*}
P_{t}^{(\epsilon)}(x) := P(\epsilon X_{t/\epsilon^\alpha} =x) , \quad \text{ for } x \in \epsilon \Z^d.
\end{equation*}

\begin{lemma} \label{lem:v-V} Assume that Assumptions \ref{cond1} and \ref{cond2} hold. For all $x\in \R^d$,
\begin{equation*}
\E \left|V_t^{(\epsilon)}(\epsilon[x/\epsilon])-v_t^{(\epsilon)}(\epsilon[x/\epsilon])\right|^m \lesssim\epsilon^{\rho m},
\end{equation*}
where $\rho:=a\wedge\eta$.
\end{lemma}
\begin{proof}
We begin by writing 
\begin{align*}
&V_t^{(\epsilon)}(\epsilon[x/\epsilon])-v_t^{(\epsilon)}(\epsilon[x/\epsilon])\\
&\quad =\sum_{k\in \Z^d} \int_0^{t-\epsilon^{\alpha}} \int_{C^{(\epsilon)}(\epsilon k)} \left[\frac{P^{(\epsilon)}_{t-s}(\epsilon k-\epsilon[x/\epsilon])}{\epsilon^d}-p_{t-s}(\epsilon k-\epsilon[x/\epsilon])\right] \sigma\big(u_s(\epsilon k)\big)\, F(\d s \, \d y). 
\end{align*}
As in the proof of the previous lemmas, we take the $m$th moment and use Burkholder's inequality. 
%
%
%
%
%
\be \begin{split} \label{eq:v-V}
&\left\| v_t^{(\epsilon)}\big(\epsilon[x/\epsilon]\big)- V_t^{(\epsilon)}\big(\epsilon[x/\epsilon]\big) \right\|_m^2   \\
& \lesssim \int_{\epsilon^\alpha}^t \d s\,\sum_{k,l \in \Z^d}\left\vert p_{s}(\epsilon k)-\epsilon^{-d}P_{s}^{(\epsilon)}(\epsilon k)\right\vert  \left\vert p_{s}(\epsilon l)-\epsilon^{-d}P_{s}^{(\epsilon)}(\epsilon l)\right\vert \int_{C^{(\epsilon)}(\epsilon k)}\d u \int_{C^{(\epsilon)}(\epsilon l)} \d v\, f(u-v).
\end{split}\ee
We have the bound
\be \label{eq:f:int}
\int_{C^{(\epsilon)}(\epsilon k)}\d u\, \int_{C^{(\epsilon)}(\epsilon l)}\d v \, f(u-v) \lesssim \frac{\epsilon^{2d}}{\epsilon^\beta +| \epsilon(k-l)|^\beta}.
\ee
We split the right hand side of \eqref{eq:v-V} as
\bes
\int_{\epsilon^\alpha}^t \d s\left[\sum_{|k|\le\frac{s^{1/\alpha}}{\epsilon},\, |l|\le\frac{s^{1/\alpha}}{\epsilon}}\cdots\, +\quad 2\sum_{|k|\le\frac{s^{1/\alpha}}{\epsilon},\, |l|>\frac{s^{1/\alpha}}{\epsilon}}\cdots \, + \sum_{|k|>\frac{s^{1/\alpha}}{\epsilon},\, |l|>\frac{s^{1/\alpha}}{\epsilon}}\cdots  \right] =: \mathcal{A}_1+2\mathcal{A}_2+\mathcal{A}_3,
\ees
where $\mathcal{A}_1,\, \mathcal{A}_2$ and $\mathcal{A}_3$ correspond to the first, second and third sums. We bound each of $\mathcal{A}_1,\, \mathcal{A}_2$ and $\mathcal{A}_3$ separately. 
 Our strategy is as follows. We will bound $\big\vert p_s(\epsilon k)- \epsilon^{-d} P_s^{(\epsilon)}(\epsilon k)\big \vert$ using Proposition \ref{prop:unif} for  $|k|\le s^{1/\alpha}/\epsilon$ and using Proposition \ref{prop:x} for $|k|> s^{1/\alpha}/\epsilon$.  We start with
\bes
\begin{split}
\mathcal{A}_1 & \lesssim \epsilon^{2d-\beta+2a}\int_{\epsilon^\alpha}^{t} \frac{\d s}{s^{2(a+d)/\alpha}} \sum_{|k|\le\frac{s^{1/\alpha}}{\epsilon},\, |l|\le\frac{s^{1/\alpha}}{\epsilon}} \frac{1}{1+|k-l|^\beta}.
\end{split}
\ees
Fixing $k$ and summing over $l$ gives us a bound of $(s^{1/\alpha}\epsilon^{-1})^{d-\beta}$. Thus the integrand is bounded by a constant times $\epsilon^{\beta-2d}s^{-(2a+\beta)/\alpha}$. We thus have
\bes
\mathcal{A}_1  \lesssim \epsilon^{2\rho}.
\ees
Next we consider $\mathcal{A}_2$.  
\bes\begin{split}
\mathcal{A}_2  &\lesssim \epsilon^{2d-\beta+2a}\int_{\epsilon^\alpha}^t \frac{\d s}{s^{(a+d-\alpha)/\alpha}} \sum_{|k|\le\frac{s^{1/\alpha}}{\epsilon},\, |l|>\frac{s^{1/\alpha}}{\epsilon}} \frac{1}{|l|^{d+\alpha+a}} \\
&\lesssim \epsilon^{2d-\beta+a} \int_{\epsilon^\alpha}^t\frac{\d s}{s^{(a+d-\alpha)/\alpha}} \cdot \frac{s^{d/\alpha}}{\epsilon^d} \cdot \frac{\epsilon^{\alpha+a}}{s^{(\alpha+a)/\alpha}}\\
&\lesssim  \epsilon^{d+\alpha-\beta+2a} \int_{\epsilon^\alpha}^t \frac{\d s}{s^{2a/\alpha}}\\
& \ll\epsilon^{2\rho}.
\end{split}\ees
Finally we bound $\mathcal{A}_3$ as follows.
\bes\begin{split}
\mathcal{A}_3 & \lesssim \epsilon^{2d-\beta+2a}\int_{\epsilon^\alpha}^t \d s\, \left(\sum_{|k|\ge \frac{s^{1/\alpha}}{\epsilon}}\frac{s}{|k|^{d+\alpha+a}} \right)^2 \\
& \lesssim \epsilon^{2d+2\alpha+4a-\beta}\int_{\epsilon^\alpha}^t\frac{\d s}{s^{2a/\alpha}}\\
&\ll \epsilon^{2\rho}.
\end{split}\ees
Combining all our bounds gives us the lemma.
\end{proof}
Next consider for $x\in \R^d$,
\bes
W_t^{(\epsilon)}(\epsilon[x/\epsilon]) := 1+\sum_{k\in \Z^d} \int_0^t\int_{C^{(\epsilon)}(\epsilon k)} \frac{P^{(\epsilon)}_{t-s}(\epsilon k-\epsilon[x/\epsilon])}{\epsilon^d} \sigma\big(u_s(\epsilon k)\big)\, F(\d s \, \d y).
\ees

\begin{lemma}
For $x\in \R^d$ and $t>0$,
\begin{equation*}  
\E\left|W_t^{(\epsilon)}(\epsilon[x/\epsilon])-V_t^{(\epsilon)}(\epsilon[x/\epsilon])\right|^m\lesssim\epsilon^{\eta m}.
\end{equation*}

\end{lemma}
\begin{proof}
We begin with
\begin{align*} 
W_t^{(\epsilon)}(\epsilon[x/\epsilon])-V_t^{(\epsilon)}(\epsilon[x/\epsilon])= \sum_{k\in \Z^d} \int_{t-\epsilon^\alpha}^t\int_{C^{(\epsilon)}(\epsilon k)} \frac{P^{(\epsilon)}_{t-s}(\epsilon k-\epsilon[x/\epsilon])}{\epsilon^d} \sigma\big(u_s(\epsilon k)\big)\, F(\d s \, \d y).
\end{align*}
As before, we have 
\begin{align*}
&\left\|W_t^{(\epsilon)}(\epsilon[x/\epsilon])-V_t^{(\epsilon)}(\epsilon[x/\epsilon])\right\|_m^2\\
&\qquad\lesssim\sum_{k,\,l\in \Z^d} \int_{t-\epsilon^\alpha}^t\int_{C^{(\epsilon)}(\epsilon k)\times C^{(\epsilon)}(\epsilon l)} \frac{P^{(\epsilon)}_{t-s}(\epsilon k-\epsilon[x/\epsilon])}{\epsilon^d} \frac{P^{(\epsilon)}_{t-s}(\epsilon l-\epsilon[x/\epsilon])}{\epsilon^d}f(y-w) \, \d s \, \d y\,\d w.
\end{align*}
Using
\begin{equation} \label{eq:f}
\int_{C^{(\epsilon)}(\epsilon k)\times C^{(\epsilon)}(\epsilon l)}f(y-w)  \, \d y\,\d w \lesssim \epsilon^{2d-\beta},
\end{equation}
and that $P^{(\epsilon)}$ are probability measures completes the proof.
\end{proof}
Before we give our final approximation lemma we state a proposition required in the proof.
\begin{proposition} The following holds uniformly in $0<\epsilon<1$
\be \label{eq:cov}
\int_0^T\d s \sum_{k,l \in \Z^d} \int_{C^{(\epsilon)}(\epsilon k)}\d u  \int_{C^{(\epsilon)}(\epsilon l)}\d v \, \frac{P_s^{(\epsilon)}(\epsilon k)}{\epsilon^d}\cdot f(u-v) \cdot \frac{P_s^{(\epsilon)}(\epsilon l)}{\epsilon^d} < \infty.
\ee
\end{proposition}
\begin{proof} Consider the regions $s\le \epsilon^\alpha$ and $s>\epsilon^\alpha$ separately. For $s\le \epsilon^\alpha$ we 
use \eqref{eq:f} which gives a bound of $\epsilon^{\alpha-\beta}$. In view of equation \eqref{eq:v-V} we need to only bound
\bes
\int_{\epsilon^\alpha}^T\d s \sum_{k,l \in \Z^d} \int_{C^{(\epsilon)}(\epsilon k)}\d u  \int_{C^{(\epsilon)}(\epsilon l)}\d v \, p_s(\epsilon k)\cdot f(u-v) \cdot p_s(\epsilon l) .
\ees
Because of Lemma \ref{scale} and Proposition \ref{prop-cor} we can bound the above by a constant multiple of 
\begin{equation*}
\epsilon^{\alpha-\beta} + \int_0^T \frac{\d s}{s^{\beta/\alpha}}
\end{equation*}
which is finite.
\end{proof}
Recall that 
\begin{equation*}
U_t^{(\epsilon)}(\epsilon[x/\epsilon])=1+\sum_{k\in \bZ^d}\int_0^t\frac{P^{(\epsilon)}(\epsilon k-\epsilon[x/\epsilon])}{\epsilon^d}\sigma(U_s^{(\epsilon)}(\epsilon k))\,\d B_s^{(\epsilon)}(k).
\end{equation*}
The above proposition implies that 
\be \label{eq:U:mom}
\sup_{0<\epsilon <1} \sup_{0\le t \le T, \, k \in \Z^d} \E\left| U_t^{(\epsilon)}(\epsilon k)\right|^m < \infty,\text{ for all } m \ge 2 \text{ and } T< \infty;
\ee
this can be seen by following the arguments in Theorem \ref{thm:dis:mom:bd}. Here is our final lemma of this section. 
\begin{lemma}\label{lem:U-W} For $x\in \R^d$ and $t>0$,
\begin{equation*}
\E\left|U_t^{(\epsilon)}(\epsilon[x/\epsilon])-W_t^{(\epsilon)}(\epsilon[x/\epsilon])\right|^m\lesssim\epsilon^{\rho m},
\end{equation*}
where $\rho= a \wedge \eta$.
\end{lemma}
\begin{proof}
We obviously have
\begin{align*}
U_t^{(\epsilon)}(\epsilon[x/\epsilon])&-W_t^{(\epsilon)}(\epsilon[x/\epsilon])\\
&=\sum_{k\in \Z^d} \int_0^t\int_{C^{(\epsilon)}(\epsilon k)} \frac{P^{(\epsilon)}_{t-s}(\epsilon k-\epsilon[x/\epsilon])}{\epsilon^d}\left[\sigma\big(U_s^{(\epsilon)}(\epsilon k)\big)-\sigma\big(u_s(\epsilon k)\big)\right]F( \d s \, \d y).
\end{align*}
We will split the integral above by using the following observation,
\begin{align*}
\sigma(U_s^{(\epsilon)}(\epsilon k))-\sigma\big(u_s(\epsilon k)\big)=\sigma\big(U_s^{(\epsilon)}(\epsilon k)\big)-\sigma(W_s^{(\epsilon)}(\epsilon k))+\sigma\big(W_s^{(\epsilon)}(\epsilon k)\big)-\sigma\big(u_s(\epsilon k)\big).
\end{align*}
From the above Lemmas, we have
\begin{align*}
\E\left|W_s^{(\epsilon)}(\epsilon k)-u_s(\epsilon k)\right|^m\lesssim\epsilon^{\rho m}.
\end{align*}
The above implies that 
\begin{align*}
\E\left|U_s^{(\epsilon)}(\epsilon k)-u_s(\epsilon k)\right|^m\leq c_2\left[\epsilon^{\rho m}+
\E\left|W_s^{(\epsilon)}(\epsilon k)-U_s^{(\epsilon)}(\epsilon k)\right|^m\right].
\end{align*}
Upon setting
\begin{align*}
\mathcal{D}^{(\epsilon)}(t):=\sup_{x\in \epsilon \Z^d}\left\{\E\left|U_t^{(\epsilon)}(x)-W_t^{(\epsilon)}(x)\right|^m\right\}^{2/m},
\end{align*}
we obtain 
\begin{align*}
\mathcal{D}^{(\epsilon)}(t)\leq c_3\epsilon^{2\rho} +c_4\int_0^t \d s \, \mathcal{D}^{(\epsilon)}(s)\sum_{k,l \in \Z^d} \int_{C^{(\epsilon)}(\epsilon k)}\d u  \int_{C^{(\epsilon)}(\epsilon l)}\d v \, \frac{P_s^{(\epsilon)}(\epsilon k)}{\epsilon^d}\cdot f(u-v) \cdot \frac{P_s^{(\epsilon)}(\epsilon l)}{\epsilon^d}. 
\end{align*}
From \eqref{eq:u:mom} and \eqref{eq:U:mom}, we have that
 $\sup_{0\le s\le T} \mathcal{D}^{(\epsilon)}(s)<\infty$.  A suitable form of Gronwall's inequality now finishes the proof.
\end{proof}
We can now finally give the main results.
\begin{proof}[Proof of Theorem \ref{thm:main}]
For the special case that the initial profile is identically one, the proof easily follows by combining the previous lemmas together with 
\begin{align*}
\E\left|u_t(\epsilon[x/\epsilon])-u_t(x)\right|^m\lesssim \epsilon^{\eta m},
\end{align*}
where $\eta$ is defined in \eqref{eq:eta:eta'}. To obtain the result in the generality as described in the introduction, it suffices to find a good bound on the following quantity
\begin{align}\label{eq:det:diff}
\Big| \int_{\R^d}p_t(x-y)u_0(y)\,\d y-\sum_{k\in  \Z^d}P^{(\epsilon)}_t(\epsilon[x/\epsilon]-\epsilon k )U_0^{(\epsilon)}(\epsilon k)\Big|\quad\text{for}\quad x\in \R^d. 
\end{align}
We begin with 
\begin{align*}
&\Big| \int_{\R^d}p_t(\epsilon[x/\epsilon]-y)u_0(y)\,\d y-\sum_{k\in  \Z^d}P^{(\epsilon)}_t(\epsilon[x/\epsilon]-\epsilon k)U_0^{(\epsilon)}(\epsilon k)\Big|\\ 
&\quad\leq \Big| \int_{\R^d}p_t(\epsilon[x/\epsilon]-y)u_0(y)\,\d y-\sum_{k\in  \Z^d}\int_{C^{(\epsilon)}(\epsilon k)}p_t(\epsilon[x/\epsilon]-\epsilon k )u_0(z)\,\d z\Big|\\
&\qquad +\Big|\sum_{k\in  \Z^d}\int_{C^{(\epsilon)}(\epsilon k)}p_t(\epsilon[x/\epsilon]-\epsilon k)u_0(z)\,\d z-\sum_{k\in  \Z^d}P^{(\epsilon)}_t(\epsilon[x/\epsilon]-\epsilon k)U_0^{(\epsilon)}(\epsilon k)\Big|\\
&\quad=:I_1+I_2.
\end{align*}
The mean value theorem and an application of Lemma \ref{ker-d} show that 
\begin{align*}
I_1&=\Big| \int_{\R^d}p_t(\epsilon[x/\epsilon]-z)u_0(z)\,\d z-\sum_{k\in  \Z^d}\int_{C^{(\epsilon)}(\epsilon k)}p_t(\epsilon[x/\epsilon]-\epsilon k)u_0(z)\,\d z\Big|\\
&\lesssim\frac{\epsilon}{t^{1/\alpha}}\int_{\R^d}p_t\left(\frac{\epsilon[x/\epsilon]-z}{2}\right)u_0(z)\,\d z.
\end{align*}
We can rewrite $I_2$ as follows,
\begin{align*}
I_2&=\epsilon^d\sum_{k\in  \Z^d}\left[p_t(\epsilon[x/\epsilon]-\epsilon k)-\frac{1}{\epsilon^d}P^{(\epsilon)}_t(\epsilon[x/\epsilon]-\epsilon k)\right]U_0^{(\epsilon)}(\epsilon k)\\
&\lesssim  \sum_{\substack{k\in  \Z^d\\ |\epsilon k|\le t^{1/\alpha}}} \frac{\epsilon^{d+a}}{t^{(d+a)/\alpha}} + \sum_{\substack{k\in  \Z^d\\ |\epsilon k|> t^{1/\alpha}}}  \frac{t\epsilon^{d+a}}{|\epsilon k|^{d+\alpha+a}} \\
&\lesssim \frac{\epsilon^a}{t^{a/\alpha}} + \frac{\epsilon^{1+a}}{t^{(1+a)/\alpha}},
\end{align*}
for $t\ge \epsilon^\alpha$. Combining these estimates and using the fact that $0<a<1$, we obtain 
\begin{align*}
\Big| \int_{\R^d}p_t(\epsilon[x/\epsilon]-y)u_0(y)\,\d y-\sum_{k\in  \Z^d}P^{(\epsilon)}_t(\epsilon[x/\epsilon]-\epsilon k)U_0^{\epsilon}(\epsilon k)\Big|\lesssim \frac{\epsilon^a}{t^{a/\alpha}},
\end{align*}
as long as $t\ge \epsilon^\alpha$. We now use the mean value theorem and the bound on the derivative of $p_t(\cdot)$ to compute
\begin{align*}
\Big| \int_{\R^d}p_t(x-y)u_0(y)\, \d y-\int_{\R^d}p_t(\epsilon[x/\epsilon]-y)u_0(y)\, \d y\Big|\lesssim \frac{\epsilon}{t^{1/\alpha}}.
\end{align*}
Thus we obtain
\begin{align*}
\Big| \int_{\R^d}p_t(x-y)u_0(y)\,\d y-\sum_{k\in  \Z^d}P^{(\epsilon)}_t(\epsilon[x/\epsilon]-\epsilon k)U_0^{(\epsilon)}(\epsilon k)\Big|\lesssim \frac{\epsilon^a}{t^{a/\alpha}},
\end{align*}
whenever $t\ge \epsilon^\alpha$. This gives the rate of convergence stated in the theorem.
\end{proof}

\section{Proof of remaining results} \label{sec:remain}
We first focus on the proof of Theorem \ref{thm:strong:approx}. Let us write
\be\begin{split} \label{u:U:split}
u_t(x) &=\int_{\R^d}p_t(x-y) u_0(y) \, \d y + \tilde{u}_t(x) ,\quad x\in \R^d\\
U_t^{(\epsilon)}(\epsilon k) &=\sum_{l\in  \Z^d}P^{(\epsilon)}_t(\epsilon l-\epsilon k )U_0^{(\epsilon)}(\epsilon l) + \tilde{U}^{(\epsilon)}(\epsilon k),\quad k \in \Z^d,
\end{split}\ee
where $\tilde{u}$ and $\tilde{U}^{(\epsilon)}$ denote the stochastic parts of $u$ and $U^{(\epsilon)}$ respectively.
We have already obtained the rate of convergence of the deterministic part of $u_t(x)-U_t^{(\epsilon)}(\epsilon[x/\epsilon])$, see the computations following \eqref{eq:det:diff}. 
We thus need to focus only the stochastic parts $\tilde{u}$ and $\tilde{U}^{(\epsilon)}$. For this we shall need the following lemma
which is a H\"older continuity estimate for the difference $U_t^{(\epsilon)}(x)- U_s^{(\epsilon)}(x)$.
\begin{lemma} \label{lem:dis:hol}  Fix $T>0$ and integer $m\ge 2$. Then 
\bes
\sup_{x\in \epsilon \Z^d} \E\left(\left\vert \tilde U_t^{(\epsilon)}(x)- \tilde U_s^{(\epsilon)}(x)\right\vert^m\right) \lesssim \left(\frac{t-s}{\epsilon^\beta}\right)^{m/2}
\ees
holds uniformly for $0<s<t \le T$ and $0<\epsilon<1$.
\end{lemma}
\begin{proof} 
Firstly \eqref{eq:U:mom} gives 
\bes
\sup_{0<\epsilon<1}\;\sup_{x\in \epsilon \Z^d,\, 0\le t\le T}  \E \big[ \big\vert \tilde U_t^{(\epsilon)}(x)\big\vert^m\big] < \infty.
\ees
Without loss of generality, let us restrict to $x=0$. An application of Burkholder's inequality gives for $0<\epsilon<\epsilon_0$
\be\begin{split} \label{eq:ho:dis:she}
 &\big\|U_t^{(\epsilon)}(0) - U_s^{(\epsilon)}(0)\big\|_m^2\\
& \lesssim \epsilon^{-2d} \sum_{k,l \in \Z^d} \int_0^s \d r \int_{C^{(\epsilon)}(\epsilon k)}\d x \int_{C^{(\epsilon)}(\epsilon l)} \d y \,\big[P_{t-r}^{(\epsilon)}(\epsilon k)-P_{s-r}^{(\epsilon)}(\epsilon k)\big] \cdot f(x-y)\cdot  \big[P_{t-r}^{(\epsilon)}(\epsilon l)-P_{s-r}^{(\epsilon)}(\epsilon l)\big] \\
&\quad  +  \epsilon^{-2d}  \sum_{k,l \in \Z^d} \int_s^t  \int_{C^{(\epsilon)}(\epsilon k)}\d x \int_{C^{(\epsilon)}(\epsilon l)} \d y \, P_{t-r}^{(\epsilon)}(\epsilon k) \cdot f(x-y) \cdot P_{t-r}^{(\epsilon)}(\epsilon l).
\end{split}\ee
We bound the second term first. Because of \eqref{eq:f} this is less than a constant multiple of
\bes
 \epsilon^{-\beta} \int_s^t \d r \, \big(\sum_k P_{t-r}^{(\epsilon)} (\epsilon k)\big)^2  \lesssim \frac{t-s}{\epsilon^\beta}.
\ees
Let us define the function $Q_r^{(\epsilon)}: \R^d \to \R_+$ by
\bes
Q_r^{(\epsilon)}(x) := P_r^{(\epsilon)}(\epsilon k) \quad \text{if } x \in C^{(\epsilon)}(\epsilon k).
\ees
The Fourier transform of this function is easily computed to be
\be\begin{split} \label{eq:qhat}
\hat{Q}_r^{(\epsilon)} (\xi) & = \mathcal{H}(\epsilon, \xi) \cdot \sum_{k\in \Z^d} P_r^{(\epsilon)}(\epsilon k) \cdot e^{i\xi\cdot \epsilon k}  \\
& = \mathcal{H}(\epsilon, \xi) \cdot  \E \exp\big(i \epsilon \xi \cdot X_{\frac{r}{\epsilon^\alpha}}\big) \\
&= \mathcal{H}(\epsilon, \xi) \cdot \exp\left[-\frac{r}{\epsilon^\alpha}[1-\hat\mu(\epsilon \xi)]\right]
\end{split}\ee
where $\mathcal{H}(\epsilon, \xi) = \int_{C^{(\epsilon)}(0)}e^{i \xi \cdot x} \d x$. It is now checked that
$\mathcal{H}(\epsilon, \xi) \le \min \big(\frac{\epsilon^{d-1}}{|\xi|} ,\,\epsilon^d\big)$.

It is known that $f(x) =(h*\tilde{h})(x)$ where 
\bes
h(x)= \frac{c_1}{|x|^{(d+\beta)/2}},\; x \in \R^d,
\ees
and $\tilde{h}(x)=h(-x)$. This can be seen for example from \eqref{eq:fourier:riesz} and by applying the Fourier transform to $h*\tilde{h}$. The first term in \eqref{eq:ho:dis:she} is equal to
\bes
\begin{split}
& \epsilon^{-2d} \int_0^s \d r \int_{\R^d}\d x \int_{\R^d} \d y \,\big[Q_{t-r}^{(\epsilon)}(x)-Q_{s-r}^{(\epsilon)}(x)\big] \cdot f(x-y)\cdot  \big[Q_{t-r}^{(\epsilon)}(y)-Q_{s-r}^{(\epsilon)}(y)\big] \\
&= C\epsilon^{-2d} \int_0^s  \d r \int_{\R}^d \d z\, \big((Q^{(\epsilon)}_{t-r}-Q^{(\epsilon)}_{s-r})*h\big)^2(z) \\
&=  C\epsilon^{-2d} \int_0^s  \d r \int_{\R}^d \d \xi\, \big(\hat{Q}^{(\epsilon)}_{t-r}-\hat{Q}^{(\epsilon)}_{s-r}\big)^2(\xi)\cdot \hat h^2(\xi) \\
&= C\epsilon^{-2d} \int_0^s  \d r \int_{\R}^d \d \xi\, \frac{ \big(\hat{Q}^{(\epsilon)}_{t-r}-\hat{Q}^{(\epsilon)}_{s-r}\big)^2(\xi)}{|\xi|^{d-\beta}}.
\end{split}
\ees
We next use our expression for $\hat{Q}_r^{(\epsilon)}$ derived in \eqref{eq:qhat} and the bound for $\mathcal{H}(\epsilon, \xi)$ to bound the above expression. We get
\bes \begin{split}
\epsilon^{-2d} \int_0^s \d r &\int_{\R^d} \d \xi \, \frac{\mathcal{H}^2(\epsilon, \xi)}{|\xi|^{d-\beta}} \cdot\exp \left(-\frac{2(s-r)[1-\hat\mu(\epsilon\xi)]}{\epsilon^\alpha}\right)\cdot \left[1-\exp\left(-\frac{(t-s)[1-\hat\mu(\epsilon\xi)]}{\epsilon^\alpha}\right)\right]^2 \\
&\lesssim\epsilon^{-2d}  \int_{\R^d} \d \xi \,\frac{\mathcal{H}^2(\epsilon, \xi)}{|\xi|^{d-\beta}} \cdot \frac{\epsilon^\alpha}{1-\hat\mu(\epsilon \xi)} \cdot \left[1-\exp\left(-\frac{(t-s)[1-\hat\mu(\epsilon\xi)]}{\epsilon^\alpha}\right)\right]^2 \\
&\lesssim\epsilon^{-2d}\cdot (t-s)\cdot  \int_{\R^d} \d \xi \,\frac{\mathcal{H}^2(\epsilon, \xi)}{|\xi|^{d-\beta}} \\
&\lesssim \epsilon^{-2} \cdot (t-s)\cdot \int_0^\infty \d r\, r^{d-1} \cdot \frac{\min(\epsilon^{2},r^{-2})}{r^{d-\beta}}\\
& \lesssim \frac{t-s}{\epsilon^\beta}.
\end{split}\ees
The third line follows from the well known inequality $1-e^{-x} \le \sqrt x$ valid for all positive $x$. In the last step we split the integral according to $r\le \epsilon^{-1}$ and $r>\epsilon^{-1}$ and bound each term separately.
\end{proof}
\begin{proof}[Proof of Theorem \ref{thm:strong:approx}] The proof is similar to the proof of Theorem 2.5 in
\cite{jose-khos-muel} and we only provide the preliminary estimates. First note that we have already shown that the difference of the deterministic parts in \eqref{u:U:split} satisfies the conclusion of the Theorem \ref{thm:strong:approx}. We thus need to concentrate on the difference of the stochastic parts in \eqref{u:U:split}.  The proof of Theorem \ref{thm:main} shows that \eqref{eq:ass} holds for the difference $\tilde u_t(x)-\tilde U_t^{(\epsilon)}(\epsilon[x/\epsilon])$ of the stochastic parts of $u$ and $U^{(\epsilon)}$. One next needs good control of the differences
$\tilde u_t(x)-\tilde u_s(x)$ and $\tilde U_t^{(\epsilon)}(\epsilon [x/\epsilon])-\tilde U_s^{(\epsilon)}(\epsilon[x/\epsilon])$. For each fixed integer $m\ge 2$ we have
\bes
 \sup_{x \in \R^d}\E \vert \tilde u_t(x) -\tilde u_s(x)\vert^m \le C \vert t-s \vert^{\tilde\eta m},
\ees
uniformly for all $0\le s,t \le T$, where
$\tilde\eta$ is defined in \eqref{eq:eta:eta'}. This H\"older continuity estimate can be found in \cite{boul-lahc-edda}. Next we apply Kolmogorov continuity theorem (\cite{spde-mini}, Theorem 4.3, page 10) to obtain the following bound
\bes
\sup_{x \in \R^d} \E \left(\sup_{\substack{0\le s,t\le T ,\\  |t-s| \le \epsilon^3 T}} \big \vert \tilde u_t(x)-\tilde u_s(x) \vert^m\right) \lesssim \epsilon^{3\tilde\eta m-4d}
\ees
valid for all large enough integers $m$. Using Lemma \ref{lem:dis:hol} and Kolmogorov continuity theorem again, we obtain a similar estimate for $\tilde U^{(\epsilon)}$
\bes
\sup_{x \in \R^d} \E \left(\sup_{\substack{0\le s,t\le T ,\\  |t-s| \le \epsilon^3 T}} \big \vert \tilde U^{(\epsilon)}_t(x)-\tilde U_s^{(\epsilon)}(x) \vert^m\right) \lesssim \epsilon^{3\tilde\eta m-4d}
\ees
uniformly for $0< \epsilon <1 $. These are the main estimates needed for the proof. The reader can consult \cite{jose-khos-muel} for the rest of the argument.
\end{proof}
\begin{proof}[Proof of Theorem \ref{thm:comp}] Theorem 1.1 in \cite{geis-mant} implies that the comparison principle holds for finite dimensional
stochastic differential equations (SDEs) of the form \eqref{eq:disc:torus}. Proposition \ref{prop:i_set:fin} and the continuity
of the solution then implies that it also holds for infinite dimensional SDEs of the form \eqref{eq:disc:she}. Finally Theorem
\ref{thm:main} and the continuity of the solution to \eqref{eq:she} proves the comparison principle for \eqref{eq:she}.
\end{proof}

\begin{proof}[Proof of Theorem \ref{thm:mom:comp}]
The proof is similar to that of Theorem 2.6 in \cite{jose-khos-muel}.
For any $\alpha>0$, we can find a random walk with generator $\mathscr{L}$ which satisfies the conditions in Assumptions \ref{cond1} and \ref{cond2}.  Theorem \ref{thm:mom:comp} then follows from Theorems \ref{thm:mom:disc:she} and \ref{thm:main}. Indeed Theorem \ref{thm:mom:disc:she} says that the comparison of moments holds for the solution $U^{(\epsilon)}$ of \eqref{eq:disc:approx} and the solution $\bar U^{(\epsilon)}$ of \eqref{eq:disc:approx} with $\sigma$ replaced by $\bar \sigma$. One then take limit as $\epsilon\downarrow 0$ and use Theorem \ref{thm:main} to obtain the comparison of moments result for $u$ and $\bar u$. 
\end{proof}

\begin{proof}[Proof of Theorem \ref{thm:lyap}]  For the Parabolic Anderson model, that is when $\sigma(x)=Cx$, Lemma 4.1 in \cite{kunw} implies that 
\bes
m^{\frac{2\alpha-\beta}{\alpha-\beta}} \lesssim \liminf_{t \to \infty} \frac{1}{t}\, \log \E\big(\big\vert u_t(x) \big\vert^m\big) \le \limsup_{t \to \infty} \frac{1}{t}\, \log \E\big(\big\vert u_t(x) \big\vert^m\big) \lesssim   m^{\frac{2\alpha-\beta}{\alpha-\beta}}.
\ees
While in \cite{kunw}, $u_0$ is assumed to be identically one, it is clear from the proof that the above continues to hold when $u_0$ is bounded away from zero and infinity. Theorem \ref{thm:lyap} thus follows immediately from Theorem \ref{thm:mom:comp}.
\end{proof}

\section{Some extensions} \label{sec:extend}
A close inspection of the proof of Theorem \ref{thm:main} indicates that one can provide several extensions. We list some of them here.
\begin{itemize}
\item It is clear that Theorem \ref{thm:main} still holds if the correlation function
$f$ behaves better than Riesz kernels in the sense that it grows slower at the origin and decays faster at infinity. 
In particular Theorem \ref{thm:main} holds if 
\bes
\vert f(x) \vert \lesssim \frac{1}{|x|^\beta},
\ees 
for some $\beta< \min(\alpha,d)$.  Examples of functions which satisfy this are the Poisson kernels, Ornstein-Uhlenbeck type
kernels, Cauchy kernels and many more; see \cite{foon-khos-col}. Furthermore we could get a faster rate of decay than that in \eqref{eq:ass} depending on how {\it nice} the function $f$ is. The corresponding comparison principles also hold for these correlation functions. This constitutes an important extension.
\item Although we have not attempted to do so, one could modify our arguments to include a drift term in \eqref{eq:she}.
\item We could also prove the results for more general initial profiles, for example unbounded functions or even nonnegative measures.
\end{itemize}


\noindent\textbf{Acknowledgements:} We thank Davar Khoshnevisan for useful discussions. M.J. was partially supported by EPSRC through grant EP/N028457/1.

{\small
\bibliographystyle{plain}
\bibliography{converge_color}
} \bigskip
\begin{small}
\noindent\textbf{Mohammud Foondun} (\texttt{mohammud.foondun@strath.ac.uk}).\\
\noindent Department of Mathematics and Statistics,
University of Strathclyde,
Glasgow, 
G1 1XH\\

\noindent\textbf{Mathew Joseph} (\texttt{m.joseph@sheffield.ac.uk}).\\
\noindent Department of Probability and Statistics, University of Sheffield,
	Sheffield, S3 7RH, UK\\
 
\noindent\textbf{Shiu-Tang Li} (\texttt{li@math.utah.edu}).\\
\noindent Department of Mathematics, University of Utah,
		Salt Lake City, UT 84112-0090, USA
		
\end{small}
\end{document}